\documentclass[10pt,twoside]{siamltex}

\usepackage{amssymb}
\usepackage{amsmath,xcolor}
\usepackage{mathrsfs}
\usepackage{graphicx}
\usepackage{calc}%
\usepackage[T1]{fontenc} 
\usepackage[varg]{txfonts}
\usepackage{cite}
\usepackage[mathscr]{euscript}
\newcommand{\Dt}{h}
\newcommand{\m}{\mathbf m }

\newcommand{\uu}{\mathbf u}
\newcommand{\vv}{\mathbf v}

\newcommand{\diver}{{\rm {div} \,\, }}
\newcommand{\varep}{\varepsilon}

\newcommand{\norm} [1]{\left\| {#1}\right\|}

\newcommand{\R} {\mathbb R}
\newcommand{\N} {\mathbb N}

\newcommand{\nablad} {\nabla\cdot}

\newcommand{\intb}[1]{\left\langle #1 \right\rangle}

\newcommand{\intd}[1]{\left( #1 \right)}
\newcommand{\esssup}{\mathop{\mathrm{ess\,sup}}}

\newcommand{\eqdef}{\overset{\mathrm{def}}{=\joinrel=}}

\def\beq{\begin{equation}}
\def\eeq{\end{equation}} 

\def\beqs{\begin{equation*}}
\def\eeqs{\end{equation*}}

\def\bals{\begin{align*}}
\def\eals{\end{align*}}

\def\bspl{\begin{split}}
\def\espl{\end{split}}

\def\myclearpage{}
\title{Existence of a solution of the mixed formulation for generalized Forchheimer flows of isentropic gases}
\author{Thinh Kieu \footnotemark[2]  }
\date{today}
\begin{document}

\maketitle
 
\renewcommand{\thefootnote}{\fnsymbol{footnote}}
\footnotetext[2]{Department of Mathematics, University of North Georgia, Gainesville Campus, 3820 Mundy Mill Rd., Oakwood, GA 30566, U.S.A. ({\tt thinh.kieu@ung.edu}).}               
               
\begin{abstract}  This paper is focused on the generalized Forchheimer flows of isentropic gas, described by a system of two nonlinear degenerating differential equations of first order. We prove the existence and uniqueness of the Dirichlet problem for stationary problem. The technique of semi-discretization in time is used to prove the existence for the time-dependent problem.
 \end{abstract}
            
            
            


\pagestyle{myheadings}
\thispagestyle{plain}
\markboth{Thinh Kieu}{Solvability of the mixed formulation for generalized Forchheimer flows of isentropic gases}

\myclearpage    
\section {Introduction}
 We consider a fluid in porous medium  occupying a bounded domain $\Omega\subset \R^d,$ $d\ge 2$ with boundary $\Gamma$. Let $x\in\R^d $, $0<T<\infty$ and $t\in (0,T]$ be the spatial and time variables respectively. The fluid flow has velocity $v(x,t)\in \R^d$, pressure $p(x,t)\in\R$ and density $  u(x,t)\in \R_+$.   

The  Darcy--Forchheimer equation,  which is considered as a momentum equation, is studied in \cite{ABHI1,HI1,HI2, Fabrie89} of the form 
\beq\label{gF}
-\nabla p =\sum_{i=0}^N a_i |v|^{\alpha_i}v. 
\eeq  

In order to take into account the presence of density in generalized Forchheimer equation, we modify \eqref{gF} using dimension analysis by Muskat \cite{Muskatbook} and Ward \cite{Ward64}. They proposed the following equation for both laminar and turbulent flows in porous media:
\beq\label{W}
-\nabla p =G\big(v^\alpha \kappa^{\frac {\alpha-3} 2} \rho^{\alpha-1} \mu^{2-\alpha}\big),\text{ where  $G$ is a function of one variable.}
\eeq 

In particular, when $\alpha=1,2$, Ward \cite{Ward64} established from experimental data that
\beq\label{FW} 
-\nabla p=\frac{\mu}{\kappa} v+c_F\frac{\rho}{\sqrt \kappa}|v|v,\quad \text{where }c_F>0.
\eeq

Combining  \eqref{gF} with the suggestive form \eqref{W} for the dependence on $\rho$ and $v$, we propose the following equation 
 \beq\label{FM}
-\nabla p= \sum_{i=0}^N a_i \rho^{\alpha_i} |v|^{\alpha_i} v,
 \eeq
where $N\ge 1,\alpha_0=0<\alpha_1<\ldots<\alpha_N$ are fixed real numbers, the coefficients $a_0(x,t), \ldots, a_N(x,t)$ are non-negative with 
$ 0<\underbar a <a_0(x,t), a_N(x,t)<\bar a<\infty,  0\le a_i(x,t) \le \bar a<\infty,\,  i=1,\ldots, N-1.$ 

Multiplying both sides of the equation \eqref{FM}  to $\rho$, we find that  
 \beq\label{eq1}
  \left(\sum_{i=0}^N a_i |\rho v|^{\alpha_i}\right) \rho v   =-\rho\nabla p.
 \eeq
 Denote the function $F:\Omega\times[0,T]\times\mathbb{R}^+\rightarrow\mathbb{R}^+$ a generalized polynomial with non-negative coefficients by
\beq\label{eq2}
F(x,t, z)=a_0(x,t)z^{\alpha_0} + a_1(x,t)z^{\alpha_1}+\cdots +a_N(x,t)z^{\alpha_N},\quad z\ge 0. 
\eeq 
The equation \eqref{eq1} can be rewritten as 
\beq\label{eq1a} 
F(x,t, |\rho v|)\rho v = -\rho\nabla p.
\eeq 
For isentropic gases, the constitutive law is
\beq\label{gas}
p=c\rho^\gamma\quad\text{for some } c,\gamma>1.
\eeq
Then from \eqref{eq1a} and \eqref{gas} follows
\beq\label{ru1}
 F(x,t, |\rho v|)\rho v   =-\rho\nabla p=-\nabla u\quad \text{with } u=\frac{c\gamma\rho^{\gamma+1}}{\gamma+1}.
\eeq

 
The continuity equation is
\beq\label{con-law}
\phi(x)\partial_t\rho+{\rm div }(\rho v)=f(x,t),
\eeq
where $\phi$ is the porosity, $f$ is external mass flow rate . 

Rewrite 
\beq\label{maineq1}
\rho=\left(\frac{\gamma+1}{c\gamma}\right)^\frac1{\gamma+1} u^\lambda \quad\text{with } \lambda=\frac1{\gamma+1}\in (0,1).
\eeq
Combining \eqref{con-law} with relation \eqref{maineq1}, we have
\beq\label{maineq2}
\phi(x) \left(\frac{\gamma+1}{c\gamma}\right)^\lambda \partial_t u^\lambda+ {\rm div }(\rho v)=f(x,t).
\eeq

By combining \eqref{ru1} and \eqref{con-law} we have
\beqs
\begin{aligned}
F(x,t,|\m|) \m = -\nabla u, \\
\phi(x) \left(\frac{\gamma+1}{c\gamma}\right)^\lambda \partial_t u^\lambda+ \diver{\m}=f(x,t),
\end{aligned}
\eeqs
where $\m=\rho v$.

By rescaling the variable $\phi(x)\to(\frac{\gamma+1}{c\gamma})^\lambda\phi(x) $. We obtain system of equations  
\beq\label{main-sys}
\begin{aligned}
F(x,t,|\m|) \m &= -\nabla u, \\
\phi(x)\partial_t u^\lambda+ \diver{\m}&=f(x,t).
\end{aligned}
\eeq  


The Darcy- Forchheimer equation in \eqref{main-sys} leads to  
 $$
 \mathcal F(|\m|)=F(x,t , |\m|) |\m| =| \nabla u|, \quad \text{ where } \mathcal F (s)=sF(s) . 
 $$
Since $\mathcal F$ is a one-to-one mapping from $[0,\infty)$ onto $[0,\infty)$, therefore one can find a unique non-negative $|\m|$ as a function of $|\nabla u|$,   
$$
|\m| = \mathcal F^{-1}(|\nabla u|). 
$$
Solving  for $\m$ from the first equation of \eqref{main-sys} gives 
\beq\label{rua} 
\m=-\frac{\nabla u}{F(x,t, \mathcal F^{-1}(|\nabla u|) )}= - K(x,t, |\nabla u |)\nabla u,
\eeq
where the function $K: \Omega\times[0,T]\times\mathbb{R}^+\rightarrow\mathbb{R}^+$ is defined for $\xi\ge 0$ by
\beq\label{Kdef}
K(x,t, \xi)=\frac{1}{F(x,t,s(x,t, \xi))},
\eeq
  with   $s=s(x,t, \xi)$  being the unique non-negative solution of  $sF(s)=\xi$.

Note that $$\mathcal F^{-1}(0)=0,\quad K(x,t,0) = \frac {1}{F(x,t,0)} = \frac 1{a_0(x,t)}>0 $$

Substituting \eqref{rua} into the second equation of \eqref{main-sys} we obtain a scalar partial differential equation (PDE) for
the density:
\beq\label{mainEq}
\phi(x)\partial_t u^\lambda- \diver{ (K(x,t, |\nabla u |)\nabla u) }=f(x,t), \quad (x,t)\in\Omega\times[0,T]. 
\eeq

From mathematical point of view, equation \eqref{mainEq} for $\lambda<1$ is a doubly nonlinear parabolic equation, which is an interesting topic of its own.
Research on doubly nonlinear parabolic equations follows the development of 
general parabolic equations \cite{LadyParaBook68,LiebermanPara96} and degenerate/singular parabolic equations \cite{DiDegenerateBook, CHK1} (see also the treaties in \cite{IvanovBook82,LadyParaBook68,LadyEllipticbook,LiebermanPara96}.) However, it requires much more complicated techniques. See monograph \cite{IvanovBook82}, review paper \cite{ ivanov1997existence, ivanov1997regularity} and references therein.

In the this paper, we focus on proving the existence of weak solutions of the system \eqref{main-sys} for the Dirichlet boundary conditions with general coefficient functions, while imposing only minimal regularity assumptions. Such a problem was not studied in the literature previously.
Our proof of solvability is based on the stationary problem first by applying the technique of the theory of nonlinear monotone operators (e.g., in \cite{BrezisMonotone,MR0259693,s97,z90}) to prove the existence  and uniqueness of a weak solution of the corresponding elliptic problem of \eqref{main-sys}. Then, using the technique of semi-discretization in time (see e.g.,  \cite{raviart69, Amirat1991, T1}), we prove the existence of weak solutions of the parabolic problem by constructing approximate solutions.  This approach can be extend in straightforward way to time-dependent nonlinear problems with degenerate coefficients or doubly nonlinear parabolic equations.

The paper is organized as follows.  
  Section \S \ref{Intrsec} contains a brief summary of some notations and the relevant results.
 In Section \S\ref{StatProb},  we consider the stationary problem of \eqref{main-sys}. The existence and uniqueness of a solution are proved in Theorem~\ref{stationaryProb}. 
Section \S\ref{SemiProb} is intended to motivate our investigation of the semi-discrete problem after discretization of the time-derivative in \eqref{main-sys} and show again the existence and uniqueness of a solution in Theorem~\ref{Sol-semidiscreteProb} . 
Section \S\ref{TransProb} is devoted to the study of  the transient problem governed by~\eqref{TProb} with homogeneous boundary conditions. We derive a priori estimates of the solutions to \eqref{semidiscrete-prob}. These are used to prove the solvability of the transient problem~\eqref{TProb}.

\section{Notations and preliminary results}\label{Intrsec}
Through out this paper, we assume that $\Omega$ is an open, bounded subset of $\mathbb{R}^d$, with $d=2,3,\ldots$, and has $C^1$-boundary $\partial \Omega$. For $s\in [0,\infty)$, we denote $L^{s}(\Omega)$ be the set of s-integrable functions on $\Omega$ and $( L^{s} (\Omega))^d$ the space of $d$-dimensional vectors which have all components in $L^{s}(\Omega)$.  We denote $(\cdot, \cdot)$ the inner product in either $L^{s}(\Omega)$ or $(L^{s}(\Omega))^d$ that is
$
( \xi,\eta )=\int_\Omega \xi\eta dx$  or $(\boldsymbol{\xi},\boldsymbol \eta )=\int_\Omega \boldsymbol{\xi}\cdot \boldsymbol{\eta} dx 
$ 
and $ \norm{u}_{L^s(\Omega)}=\left(\int_{\Omega} |u(x)|^s dx\right)^{1/s}$ for standard  Lebesgue norm of the measurable function. 
The notation $\intb{\cdot ,\cdot}$ will be used for the $L^s(\partial \Omega)$ inner-product. 
 For $m\ge 0, s\in [0,\infty]$, we denote the Sobolev spaces by $W^{m,s}(\Omega)=\{v \in L^s(\Omega),: D^{\alpha} v \in L^s(\Omega), |\alpha|\le m \}$ and the norm of $W^{m,s}(\Omega)$ by $
 \norm{v}_{W^{m,s}(\Omega)} = \left(\sum_{|\alpha|\le m}\int_{\Omega} |D^\alpha u|^s dx\right)^{1/s},\text { and }\norm{v}_{W^{m,\infty}(\Omega)} = \sum_{|\alpha|\le m}\esssup_{\Omega} |D^\alpha u|.$ The trace operator $\zeta : W^{m,s}(\Omega)\to W^{m-1/s,s}(\partial \Omega)$ is onto. We denote by $W_0^{m,s}(\Omega)$ the kernel of $\zeta$  and its dual space by $W^{-m,s^*}(\Omega)= (W_0^{m,s}(\Omega))'$, where $1/s+1/s^*=1$. The test space $\mathcal D(\Omega):=C_0^\infty(\Omega)$ is dense subset of $L^s(\Omega)$ and of $W_0^{m,s}(\Omega)$ and that of $\mathcal D(\bar \Omega):=\{\varphi, \varphi\in \mathcal  D(\R^n)  \}$ is a dense subspace of $W^{m,s}(\Omega)$. Finally we define $L^s(0, T;X)$ to be the space of all measurable functions $v : [0, T]\to X$ with the norm $\norm{v}_{L^s(0,T;X)}=\left(\int_0^T \norm{v(t)}_X^s dt\right)^{1/s}$, and $L^\infty(0, T;X)$ to be the space of all measurable functions $v: [0, T] \to X$ such that $v: t\to \norm{v(t)}_X$ is essentially bounded on $[0, T]$ with the norm $\norm{v}_{L^\infty(0,T;X)}=\esssup_{t\in[0,T]}\norm{v(t)}_X$.

Our calculations frequently use the following exponents
\begin{align}
\label{a-const }
   s= \alpha_N+2,\quad  \alpha=\frac{\alpha_N}{\alpha_N+1}, \quad s^*=2-\alpha=\frac{s}{s-1},  \\
  \label{Bdef}
 r=1+\lambda\in (1,2), \quad  r^*=\frac{\lambda+1}{\lambda}=1+\frac{1}{\lambda}\in (2, \infty).
 \end{align}
The arguments $C, C_1,C_2\ldots$ represent for positive generic constants and their values depend on exponents, coefficients of polynomial  $F$,  the spatial dimension $d$ and domain $\Omega$, independent of the initial and boundary data and time step. These constants may be different place by place. 
 
We introduce the space $W(\rm{div}; \Omega)$ defined by
 $
 W(\rm{div}; \Omega) =\left\{ \vv\in (L^s(\Omega))^d , \nabla\cdot\vv \in L^{ r^*}(\Omega)\right\} 
 $ 
 and equipped it with the norm
 $
 \norm{\vv}_{W(\rm{div}; \Omega)} = \norm{\vv}_{L^s} + \norm{\nabla\cdot \vv}_{L^{ r^*}}.
 $ 
Since $W(\rm{div};\Omega)$ is a closed subspace of $(L^s(\Omega))^{d}$, it follows that $W(\rm{div};\Omega)$ is a reflexive Banach space; the boundary $\vv\cdot\nu|_{\partial\Omega}$ exist and belong to $W^{-1/s,s^*}(\partial\Omega)$ and we have the Green's formula 
\beq
\int_\Omega \vv\nabla\psi dx+\int_\Omega \psi \nabla\cdot \vv dx= \int_{\partial\Omega} \psi \vv\cdot\nu d\sigma
\eeq  
hold for every $\vv\in W(\rm{div}; \Omega)$ and $\psi\in (W(\rm{div}; \Omega))'$ (see Lemma 3 in \cite{FM77}).

The function $F(\cdot)$ has the following properties.
\begin{lemma}The following inequality hold for all $y',y\in \R^d$ 
\begin{align}
\label{F-cont}
(i)\qquad &\left|F(x,t,|y'|)y'- F(x,t,|y|)y \right| \le C_1 \left(1+ |y'|^{\alpha_N}+ |y|^{\alpha_N} \right)|y'-y|.\\
\label{F-mono}
(ii)\qquad &\left(F(x,t,|y'|)y'- F(x,t,|y|)y \right)\cdot (y'-y) \ge  C_2 \left( |y'-y|^2 + |y'-y|^s\right),
\end{align}
where the constants $C_1(N,\bar a, \rm {deg} (F))>0$, and $ C_2(N,\underline{a}, \rm {deg} (F) )>0.$ 
\end{lemma}
\begin{proof}

(i)  Let $\gamma(t)=\tau y'+ (1-\tau)y, \tau\in [0,1]$ and $h(t) =F(x,t,|\gamma(\tau)|)\gamma(\tau)$.   Then 
\begin{align*}
&\left|F(x,t,|y'|)y'- F(x,t,|y|)y\right|= |h(1)-h(0)| =\left|\int_0^1 h'(\tau) d\tau\right|\\
 &\qquad=\left|\int_0^1 F(x,t,|\gamma(\tau)|)(y'-y) + F_z(x,t,|\gamma(\tau)|)\frac{\gamma(\tau)(y'-y)}{|\gamma(\tau)|}\gamma(\tau)   d\tau \right|\\
&\qquad\le |y'-y|\int_0^1 F(x,t,|\gamma(t)|) + F_z(x,t, |\gamma(\tau)|)|\gamma(\tau)|  d\tau.
\end{align*}

Note that $ F_z(x,t, |\gamma(\tau)|)|\gamma(\tau)| =  \sum_{i=0}^N a_i\alpha_i |\gamma(\tau)|^{\alpha_i} \le \alpha_N F(x,t, |\gamma(\tau)|) $
thus 
\beqs
|F(x,t,|y'|)y'- F(x,t,|y|)y| \le(1+\alpha_N) |y'-y|\int_0^1 F(x,t,|\gamma(\tau)|)  d\tau. 
\eeqs
 Using the inequality $x^\beta \le 1 +x^\gamma$ for $x\ge 0,  0<\beta<\gamma$ we find that  
 \beqs
 F(x,t,s) \le \max_{i=0,\ldots, N}a_i(x,t)\sum_{i=0}^N 1+s^{\alpha_N}\le (N+1) \max_{i=0,\ldots, N} a_i(x,t)(1+s^{\alpha_N}). 
 \eeqs
 Thus
\begin{align*}
|F(x,t, |y'|)y'- F(x,t,|y|)y| &\le(1+\alpha_N)(N+1) \max_{i=0,\ldots, N} a_i |y'-y|\Big(1+ \int_0^1 |\gamma(\tau)|^{\alpha_N}  d\tau\Big)\\
&\le (1+\alpha_N)(N+1) \max_{i=0,\ldots, N}a_i |y'-y|\Big(1+ \int_0^1 (|y'|+ |y|)^{\alpha_N}  d\tau\Big) \\
&\le 2^{\alpha_N}(1+\alpha_N)(N+1) \max_{i=0,\ldots, N} a_i \Big(1+ |y'|^{\alpha_N}+ |y|^{\alpha_N} \Big)|y'-y|,
\end{align*}
which proves \eqref{F-cont} hold. 
 
(ii)  Let $k(\tau) =F(x,t,|\gamma(\tau)|)\gamma(\tau)(y'-y)$.  Then 
\begin{align*}
&(F(x,t,|y'|)y'- F(x,t,|y|)y)\cdot(y'-y)= k(1)-k(0) =\int_0^1 k'(\tau) d\tau\\
&\qquad =\int_0^1\Big (F(x,t, |\gamma(\tau)|)|y'-y|^2 + F_z(x,t, |\gamma(\tau)|)\frac{|\gamma(\tau)(y'-y)|^2}{|\gamma(\tau)|}\Big) d\tau.
\end{align*}
Note that 
\[
\frac{|\gamma(\tau)(y'-y)|^2}{|\gamma(\tau)|}=\frac{|\gamma(\tau)|^2|y'-y|^2 \cos^2(\beta(t))}{|\gamma(\tau)|}=|\gamma(t)||y'-y|^2 \cos^2(\beta(t)),  
\]
where $\beta(t)$ is angle between $\gamma(t)$ and $y'-y$. 
 It  implies that  
\begin{align*}
&(F(x,t,|y'|)y'- F(x,t,|y|)y)\cdot(y'-y)\\
&\qquad=|y'-y|^2\int_0^1 \Big(F(x,t,  |\gamma(\tau)|) + F_z(x,t,|\gamma(\tau)|)|\gamma(\tau)|\cos^2(\beta(t)\Big) d\tau \\
&\qquad \ge |y'-y|^2\int_0^1 (1+\alpha_1\cos^2(\beta(t)) F(x,t,|\gamma(\tau)|) d\tau\\
&\qquad \ge |y'-y|^2\Big(a_0+a_N\int_0^1 |\gamma(t)|^{\alpha_N} dt\Big).
\end{align*}
The two last inequality are obtained by using the inequalities 
\beqs
 F_z(x,t, |\gamma(\tau)|)|\gamma(\tau)| \ge \alpha_1 F(x,t, |\gamma(\tau)|)\quad \text { and } \quad F(x,t,|\gamma(\tau)|)\ge a_0+a_N|\gamma(\tau)|^{\alpha_N}.
\eeqs
It is proved (see e.g in \cite{CHIK1} Lemma 2.4) that 
\beqs
\int_0^1 |\gamma(t)|^{\alpha_N} dt \ge \frac{|y'-y|^{\alpha_N}}{2^{\alpha_N+1}(\alpha_N+1) }. 
\eeqs
Hence
\beqs
(F(x,t,|y'|)y'- F(x,t,|y|)y)\cdot(y'-y)\ge (1+\alpha_1)|y'-y|^2\left(a_0+a_N\frac{|y'-y|^{\alpha_N}}{2^{\alpha_N+1}(\alpha_N+1) }\right).
\eeqs
The proof is complete.
\end{proof}

We recall some elementary inequalities that will be used in this paper.
\begin{lemma} The following inequality hold for all $a,b\ge 0$, $\lambda\in (0,1]$. 
\begin{align}
\label{ee3}
 & \frac{a^p+b^p}2\le (a+b)^p\le 2^{|p-1|}(a^p+b^p)\quad  \text{for all } p>0. \hspace{2cm}\\
\label{H-cont} 
 &\left|a^\lambda - b^\lambda \right| \le |a-b|^\lambda. \hspace{2cm} \\
\label{Mono}
&\frac{|a-b|^2}{|a|^{1-\lambda}+|b|^{1-\lambda}}\le ( a^\lambda -b^\lambda)(a-b). \hspace{2cm}\\
\label{ineqa}
&(a^\lambda -b^\lambda)a\ge \frac{\lambda}{1+\lambda}(a^{ 1+\lambda } -b^{ 1+\lambda})=\frac 1{r^*}(a^{r} -b^{r}) . \hspace{2cm}
\end{align}

\end{lemma}

\section {The steady-state problem}\label{StatProb}
We consider the stationary problem governed by the Darcy-Forchheimer equation and the stationary continuity equation together with Dirichlet boundary condition   
\beq\label{stationaryProb}
\begin{aligned}
F(x,|\m|) \m = -\nabla u\qquad x\in \Omega,\\
\diver{\m}=f(x) \qquad x\in \Omega,\\
u =-u_b(x)\qquad  x\in\partial \Omega.
\end{aligned}
\eeq
\subsection{The mixed formulation of the stationary problem}
The mixed formulation of \eqref{stationaryProb} reads as follows.
 Find $(\m,u)\in W(\rm {div}; \Omega)\times L^{ r}(\Omega)$ such that  
\beq\label{WeakStationaryProb}
\begin{aligned}
(F(x,|\m|) \m, \vv) -(  u, \nabla\cdot \vv)=- \langle  u_b, \vv\cdot \nu  \rangle \quad \text{ for all } \vv\in W(\rm {div}; \Omega),\\
(\nabla\cdot \m, q)=(f,q) \quad\text{ for all } q\in L^{ r}(\Omega).
\end{aligned}
\eeq
We introduce a bilinear form $b:W(\rm {div}; \Omega)\times L^{ r}(\Omega)\to \R$ by mean of 
\beqs
b(\vv,q)=  (\nabla \cdot \vv, q) \quad \text{ for all }  \vv\in W(\rm {div}; \Omega), q\in L^{ r}(\Omega),
\eeqs 
 and a nonlinear form $a: (L^{s}(\Omega))^d\times (L^{s}(\Omega) )^d\to \R$ by mean of  \beqs
 a(\uu,\vv)= (F(x,t,|\uu|)\uu, \vv ) \quad \text { for all } \uu,\vv\in (L^{s}(\Omega))^d. 
\eeqs
Then we rewrite the mixed formulation \eqref{WeakStationaryProb} as follows. Find $(\m,u)\in W(\rm div, \Omega)\times L^{ r}(\Omega)\equiv V\times Q$ such that  
\beq\label{equivform}
\begin{aligned}
a(\m, \vv) -b(\vv,u) =- \intb{u_b, \vv\cdot \nu } \quad \text{ for all } \vv\in W(\rm{div}; \Omega),\\
b(\m, q)=(f,q) \quad \text{ for all } q\in Q.
\end{aligned}
\eeq 
\subsection{Existence results} This subsection is devoted to establish  the existence and uniqueness of weak solution of the stationary problem~\eqref{stationaryProb}.   
 \begin{theorem}\label{SolofStationaryProb}
Suppose
$f\in L^{ r^*}(\Omega),$ and $u_b\in W^{1/s, s}  (\partial \Omega)$. The mixed formulation \eqref{WeakStationaryProb} of the stationary problem \eqref{stationaryProb} has a unique solution $(\m,u)\in W(\rm{div};\Omega)\times L^{ r}(\Omega)$.  
\end{theorem}
\begin{proof}
 We use regularization to show the existence of a weak solution $(\m, u)\in V\times Q$ to problem~\eqref{WeakStationaryProb}. The proof will be divided into four steps.  
In step 1, we introduce an approximate problem. 
In step 2 we show that the approximate solution $(\m_\varep, u_\varep)$ is bounded independence of $\varep$.  
In step 3 we prove the limit $(\m, u)$ of the approximate solution $(\m_\varep, u_\varep)$ satisfying problem \eqref{WeakStationaryProb}. 
Step 4 is devoted to prove the uniqueness of weak solution $(\m, u)$ to the  problem \eqref{WeakStationaryProb}. 

{\bf Step 1. } For the fixed $\varep>0$, we consider the following regularized problem. Find $(\m_\varep,u_\varep)\in V \times Q$ such that  
\beq\label{reg-prob}
\begin{aligned}
a(\m_\varep, \vv)+ \varep (|\nabla\cdot \m_\varep|^{ r^*-2} \nabla\cdot \m_\varep ,\nabla\cdot \vv)  - b(\vv,u_\varep) =- \intb{u_b,\quad \vv\cdot \nu } \quad  \text{ for all } \vv\in V,\\
\varep(u^\lambda_\varep, q)+ b(\m_\varep, q)=(f,q) \quad  \text{ for all } q\in Q.
\end{aligned}
\eeq 
\begin{lemma}\label{StatSol}
For every $\varep>0$, there is a unique solution $(\m_\varep, u_\varep)\in V\times Q$ of the regularized  problem \eqref{reg-prob}.  
\end{lemma}
\begin{proof}
Adding the left hand side of \eqref{reg-prob}, we obtain the nonlinear form defined on $V\times Q$, 
\beq\label{a-eps}
a_\varep((\m_\varep,u_\varep), (\vv,q) ) \eqdef a(\m_\varep,\vv) + \varep(|\nabla\cdot \m_\varep|^{ r^*-2} \nabla\cdot \m_\varep ,\nabla\cdot \vv) - b(\vv, u_\varep)+\varep(u^\lambda_\varep,q)+b(\m_\varep,q)\text{ for all } (\vv, q)\in V\times Q. 
\eeq
A nonlinear operator $\mathcal A_\varep: V\times Q \to (V\times Q)'$ defined by 
\beqs
\intb{\mathcal A_\varep((\uu,p)), (\vv,q) }_{(V\times Q)'\times (V\times Q)} = a_\varep((\uu,p), (\vv,q)). 
\eeqs
Then 
$\mathcal A_\varep$ is continuous, coercive and strictly monotone. 

Applying the theorem of Browder and Minty (see in \cite{zeidler1989}, Thm. 26.A) for every $\tilde f\in (V\times Q)'$, there exists unique a solution $(\m_\varep, u_\varep)\in V\times Q$ of the operator equation $\mathcal A_\varep (\m_\varep,  u_\varep) = \tilde f$. In particular, we choose the linear form $\tilde f$ defined by $\tilde f (\vv, q) :=  -\langle   u_b, \vv\cdot \nu  \rangle + (f,q)$, which arises by adding the right hand sides of \eqref{reg-prob}.
Therefore \eqref{reg-prob} has a unique solution.

What is left is to show that  $\mathcal A_\varep$ is continuous, coercive and strictly monotone.
   
     For the continuity,  
\begin{multline}\label{est0}
\intb{\mathcal A_\varep((\uu_1,p_1)-\mathcal A_\varep((\uu_2,p_2)), (\vv,q) }_{(V\times Q)'\times (V\times Q)}
=a(\uu_1,\vv)-a(\uu_2,\vv) \\
+ \varep(|\nabla\cdot \uu_1|^{ r^*-2} \nabla\cdot \uu_1-|\nabla\cdot \uu_2|^{ r^*-2} \nabla\cdot \uu_2,\nabla\cdot\vv)
 +\varep(p^\lambda_1-p^\lambda_2,q)
  - b(\vv, p_1-p_2)+b(\uu_1-\uu_2,q).
\end{multline}
By \eqref{F-cont} and using H\"older's inequality  
\beqs
\begin{split}
a(\uu_1,\vv)-a(\uu_2,\vv) &\le \intd{ (1+|\uu_1|^{s-2}+|\uu_2|^{s-2})|\uu_1-\uu_2|, |\vv|}\\
&\le C \left(1+\norm{\uu_1}_{L^s}^{s-2}+ \norm{\uu_2}_{L^s}^{s-2}\right)\norm{\uu_1-\uu_2}_{L^s}.
\end{split}
\eeqs
On account of \eqref{H-cont} and using H\"older's inequality 
\begin{align*}
&\intd{p^\lambda_1-p^\lambda_2,q}\le (|p_1-p_2|^\lambda, q)\le \norm{p_1-p_2}_{L^{ r}}^\lambda\norm{q}_{L^{ r}},\\
&b(\vv, p_1-p_2) \le \norm{\nabla\cdot \vv}_{L^{ r^* } }\norm{p_1-p_2}_{L^{ r}}, \quad \text{ and }\quad 
b(\uu_1-\uu_2, q) \le \norm{\nabla\cdot(\uu_1-\uu_2)}_{L^{ r^* } }\norm{q}_{L^{ r}}.
\end{align*}
From \eqref{F-cont}, we find that 
\begin{align*}
\intd{|\nabla\cdot \uu_1|^{ r^*-2} \nabla\cdot \uu_1-|\nabla\cdot \uu_2|^{ r^*-2} \nabla\cdot \uu_2,\nabla\cdot\vv}&\le C\intd{|1+|\nabla\cdot \uu_1|^{ r^*-2}+|\nabla\cdot \uu_2|^{ r^*-2}|\cdot |\nabla\cdot (\uu_1- \uu_2)|,|\nabla\cdot\vv|}\\
&\le C\left(1+\norm{\nabla\cdot \uu_1}_{L^{ r^*}}^{ r^*-2}+\norm{\nabla\cdot \uu_2}_{L^{ r^*}}^{ r^*-2}\right) \norm{\nabla\cdot (\uu_1- \uu_2)}_{L^{ r^*}}\norm{\nabla\cdot\vv}_{L^{ r^*}}.
\end{align*}
From the above it follows that
\begin{multline*}
\intb{\mathcal A_\varep((\uu_1,p_1)-\mathcal A_\varep((\uu_2,p_2), (\vv,q) )}_{(V\times Q)'\times (V\times Q)}
\le  C_\varep\left( 1+\norm{\uu_1}_{L^s}^{s-2}+\norm{\uu_2}_{L^s}^{s-2}\right)\\
(1+\norm{\nabla\cdot \uu_1}_{L^{ r^*}}^{ r^*-2}+\norm{\nabla\cdot \uu_2}_{L^{ r^*}}^{ r^*-2})
\left(\norm{\uu_1-\uu_2}_V+\norm{p_1-p_2}_Q +\norm{p_1-p_2}_Q^\lambda \right)\left( \norm{\vv}_V+ \norm{q}_Q\right),
\end{multline*}
for all $\vv\in V, q\in Q.$  This yields 
\begin{multline*}
\norm{\mathcal A_\varep((\uu_1,p_1)-\mathcal A_\varep((\uu_2,p_2)}_{(V\times Q)'}
\le C_\varep\left( 1+\norm{\uu_1}_{L^s}^{s-2}+\norm{\uu_2}_{L^s}^{s-2}\right)\\
(1+\norm{\nabla\cdot \uu_1}_{L^{ r^*}}^{ r^*-2}+\norm{\nabla\cdot \uu_2}_{L^{ r^*}}^{ r^*-2})
\left(\norm{\uu_1-\uu_2}_V+\norm{p_1-p_2}_Q +\norm{p_1-p_2}_Q^\lambda \right).
\end{multline*}

For $\mathcal A_\varep$ is the coercive.     
\beqs
\begin{split}
\intb{\mathcal A_\varep(\uu,p), (\uu,p) }_{(V\times Q)'\times (V\times Q)}
&\ge C \left( \norm{\uu}^2+  \norm{\uu}_{L^s}^{s}\right) +\varep \left( \norm{\nabla \cdot\uu}_{L^{ r^*}}^{ r^*}  +\norm{p}_Q^{ r}\right)\\
&\ge C_{\varep} \left(  \norm{\uu}_{L^s}^{s} +\norm{\nabla \cdot\uu}_{L^{ r^*}}^{ r^*}  +\norm{p}_Q^{ r}\right), 
\end{split}
\eeqs  
whence 
\beqs
\frac{\intb{\mathcal A_\varep(\uu,p), (\uu,p) }_{(V\times Q)'\times (V\times Q)}}{ \norm{\uu}_V +\norm{p}_Q }
\ge C_{\varep} \frac{\norm{\uu}_{L^s}^{s} +\norm{\nabla \cdot\uu}_{L^{ r^*}}^{ r^*}  +\norm{p}_Q^{ r}}{\norm{\uu}_{L^s} +\norm{\nabla\cdot \uu}_{L^{ r^*}} +\norm{p}_Q  }. 
\eeqs  
Therefore we deduce that 
\beqs
\lim_{\norm{(\vv,q)}_{V\times Q}\to +\infty} \frac{\intb{\mathcal A_\varep(\uu,p), (\uu,p) }_{(V\times Q)'\times (V\times Q)}}{ \norm{\uu}_V +\norm{p}_Q } =+\infty.  
\eeqs  

For $\mathcal A_\varep$ is the strictly monotone. 
\begin{multline*}
\intb{\mathcal A_\varep (\uu,p)-\mathcal A_\varep(\vv,q), (\uu-\vv,p-q) }_{(V\times Q)'\times (V\times Q)}\\
\ge C(\norm{\uu-\vv}^2 +\norm{\uu-\vv}_{L^s}^s ) + \varep\left(\norm{\nabla\cdot (\uu-\vv)}_{L^{ r^*}}^{ r^*} +\norm{p-q}_{Q}^{ r}\right)\\
 \ge C_{\varep}  \left( \norm{\uu-\vv}_{L^s}^s +\norm{\nabla\cdot(\uu-\vv)}_{L^{ r^*}}^{ r^*}  +\norm{p-q}_Q^{ r}\right) >0 \quad  \text{ for all } (\uu,p)\neq(\vv, q).    
\end{multline*}
\end{proof}

{\bf Step 2.} 
Next, we show that the solution $(\m_\varep,  u_\varep)$ is bounded independently of $\varep$. To do this, we use the following result (see in \cite{PG16} Lemma A.3 or \cite {Sandri1993}~Lemma A.1),
\begin{lemma}\label{dualnorm}
Let $s>1$ and $1/s+1/s^*=1$. Then there exists a constant $C_*>0$ such that
\beq\label{supinfcdn}
C_*\norm{q}_{L^{s^*}}\le \sup_{\vv \in W(\rm{div},\Omega)} \frac{b(\vv,q)}{\norm{\vv}_{ W(\rm{div},\Omega)} } \quad \text{ for all } \vv\in  W(\rm{div},\Omega), q\in L^s(\Omega).
\eeq 
\end{lemma} 
\begin{lemma}\label{stationary-sol-indep-eps}
There exists $\mathcal C>0$ independent of $\varep$ such that for sufficiently small $\varep>0$ the solution $(\m_\varep, u_\varep)$ of \eqref{reg-prob} satisfies the following estimates
\beq\label{uQmV}
\norm{ u_\varep}_Q+ \norm{\m_\varep}_V\le \mathcal C.
\eeq
\end{lemma}
\begin{proof}
We begin with a bound for the norm of $\nabla\cdot\m_\varep$. Using the second equation of \eqref{reg-prob} with $q={\rm sgn}(\nabla\cdot \m_\varep) |\nabla\cdot \m_\varep|^{{ r^* - 1}}$, we obtain
\beqs
\norm{\nabla\cdot \m_\varep}_{L^{ r^*}}^{ r^*} 
\le  \norm{f}_{L^{ r^*}}\norm{\nabla\cdot \m_\varep}_{L^{ r^*}}^{{ r^* - 1}} + \varep \norm{ u_\varep}_{L^{ r}}^{\lambda} \norm{\nabla\cdot \m_\varep}_{L^{ r^*}}^{{ r^* - 1}}.
\eeqs
It implies that
\beq\label{bound-divm}
\norm{\nabla\cdot \m_\varep}_{L^{ r^*}} \le  \norm{f}_{L^{ r^*}} + \varep \norm{ u_\varep}_Q^{\lambda}.
\eeq
Taking the test function $(\vv, q)=(\m_\varep,  u_\varep)$ in \eqref{reg-prob} gives
\beq\label{estRHS1}
\begin{aligned}
&a(\m_\varep, \m_\varep)+\varep(|\nabla\cdot \m_\varep|^{ r^*-2} \nabla\cdot \m_\varep, \nabla\cdot\m_\varep)+ \varep( u_\varep^\lambda,  u_\varep)\\
&\qquad=- \langle   u_b, \m_\varep\cdot \nu  \rangle + (f, u_\varep)  =-\intd{ \nabla \cdot \m_\varep,   u_b }-\intd{ \nabla u_b , \m_\varep}+ \intd {f, u_\varep} \\
&\qquad\le \norm{ u_b}_{V'} \left(\norm{\m_\varep}_{L^s} + \norm{\nabla \cdot \m_\varep}_{L^{ r^*}} \right)  +\norm{f}_{L^{ r^*}}\norm{ u_\varep}_{L^{ r}} .
\end{aligned}
\eeq
Using \eqref{bound-divm} and the fact that $a(\m_\varep, \m_\varep)\ge C(\norm{\m_\varep}_{L^s}^{s}+\norm{\m_\varep}_{L^2}^{2}) $ , we may conclude that 
\beq\label{bound-mvarep}
C\norm{\m_\varep}_{L^s}^{s} +\varep\left(\norm{\nabla \cdot \m_\varep}_{L^{ r^*}}^{ r^*} +\norm{ u_\varep}_Q^{ r}\right)
\le  \norm{ u_b}_{V'} \left(\norm{\m_\varep}_{L^s} +\norm{f}_{L^{ r^*}} + \varep \norm{ u_\varep}_Q^{\lambda} \right)  +\norm{f}_{L^{ r^*}}\norm{ u_\varep}_{Q}. 
\eeq
To bound $ u_\varep$ we employ the inf-sup condition \eqref{supinfcdn}. The first equation of \eqref{reg-prob} and the above estimate for $\norm{ \nabla\cdot\m_\varep}_{L^{ r^*}}$, we have 
\beqs
\begin{split}
&C_*\norm{ u_\varep}_Q\le \sup_{\vv\in V} \frac{b(\vv, u_\varep)}{\norm{\vv}_V} 
= \sup_{\vv\in V} \frac{a(\m_\varep,\vv) +\varep(\nabla\cdot \m_\varep|^{ r^*-2} \nabla\cdot \m_\varep,\nabla\cdot\vv)+ \langle   u_b, \vv\cdot \nu  \rangle}{\norm{\vv}_V}\\
&\quad\le\sup_{\vv\in V} \frac{C(\norm{\m_\varep}_{L^s}+\norm{\m_\varep}_{L^s}^{s-1} ) \norm{ \vv}_{L^s}+\varep\norm{\nabla\cdot \m_\varep}_{L^{ r^*}}^{{ r^* - 1}}\norm{\nabla\cdot \vv}_{L^{ r^*}} +\norm{ u_b}_{V'} (\norm{\vv}_{L^s} + \norm{\nabla \cdot \vv} )   }{\norm{\vv}_V}\\
&\quad\le C\left(\norm{\m_\varep}_{L^s}+\norm{\m_\varep}_{L^s}^{s-1}\right ) +\varep\Big( \norm{f}_{L^{ r^*}} + \varep \norm{ u_\varep}_Q^{\lambda} \Big)^{{ r^* - 1}} +\norm{ u_b}_{V' }\\
&\quad\le C\left(\norm{\m_\varep}_{L^s}+\norm{\m_\varep}_{L^s}^{s-1}\right ) +\varep 2^{ r^*-2}  \norm{f}_{L^{ r^*}}^{{ r^* - 1}} + \varep^{2} 2^{ r^*-2} \norm{ u_\varep}_Q +\norm{ u_b}_{V' }
\end{split}
\eeqs
for some constant $C_*>0$. Hence, for sufficiently small $\varep$ (e.g.,  $\varep\le (2^{1- r^* }C_*)^{1/2}$), 
\beq\label{bound-rhovarep}
\norm{ u_\varep}_Q\le C\left(\norm{\m_\varep}_{L^s}+\norm{\m_\varep}_{L^s}^{s-1}  +\norm{f}_{L^{ r^*}}^{{ r^* - 1}} +  \norm{ u_b}_{V' }\right). 
\eeq  
Substituting \eqref{bound-rhovarep} into \eqref{bound-mvarep} leads to 
\begin{align*}
 \norm{\m_\varep}_{L^s}^{s} 
\le C\norm{ u_b}_{V'} \left(\norm{\m_\varep}_{L^s} +\norm{f}_{L^{ r^*}}\right)+ C\left(\norm{ u_b}_{V'}+\norm{f}_{L^{ r^*}}\right)
\left( \norm{\m_\varep}_{L^s}^{s-1}+\norm{\m_\varep}_{L^s}  +\norm{f}_{L^{ r^*}}^{{ r^* - 1}} +  \norm{ u_b}_{V' }+1\right).
\end{align*}
Then by using Young's inequality, we obtain
\beq\label{mLsbound}
 \norm{\m_\varep}_{L^s}^{s} \le C_1,
\eeq
where 
$
C_1= C\left(\norm{ u_b}_{V'} + \norm{f}_{L^{ r^*}}\right)\big(\norm{ u_b}_{V'} + \norm{f}_{L^{ r^*}}^{{ r^* - 1}}\big) +\left(\norm{ u_b}_{V'} + \norm{f}_{L^{ r^*}}\right)^s+1.
$

Insert \eqref{mLsbound} into \eqref{bound-rhovarep} yields 
\beq\label{ub1}
\norm{ u_\varep}_Q\le C_2, 
\eeq
where
$
 C_2=C_1^{1/s}+C_1^{(s-1)/s}  +\norm{f}_{L^{ r^*}}^{{ r^* - 1}} +  \norm{ u_b}_{V' }. 
$
Using this estimate in \eqref{bound-divm} yields 
\beqs
\norm{\nabla\cdot \m_\varep}_{L^{ r^*}} \le  \norm{f}_{L^{ r^*}} +  C_2^{{ r^* - 1}} .
\eeqs
Therefore 
\beq\label{mb1}
\norm{\m_\varep}_V\le  C_1+ C_2^{{ r^* - 1}}.
\eeq
 The  assertion of the lemma follows from \eqref{ub1} and \eqref{mb1}. 
\end{proof}

{\bf Step 3.} Adding the left hand side of \eqref{WeakStationaryProb}, we obtain the following nonlinear form defined on $V\times Q$ by
\beqs
a( (\m, u), (\vv,q)  ):= a(\m, \vv) - b(\vv, u) + b(\m, q).
\eeqs
Consider the nonlinear operator $\mathcal A: V\times Q \to (V\times Q)'$ defined by 
\beqs
\intb{\mathcal A(\uu,p), (\vv,q)}_{ (V\times Q)' \times (V\times Q) }:=a( (\uu, p), (\vv,q)  ).
\eeqs
Set $\varep=1/n$, and let $(\m_n,   u_n)$ be the unique solution of the regularized problem \eqref{reg-prob}. Since $(\m_n,  u_n)$ is a bounded sequence in $V\times Q,$ there exists a weakly convergent subsequence, again denoted by $(\m_n,  u_n)$, with weak limit $(\m, u)\in V\times Q.$ For  $\tilde f(\vv,q) :=- \langle   u_b, \vv\cdot \nu  \rangle +(f,q) \in (V\times Q)',$   
\beq
\begin{split}
\norm{\mathcal A(\m_n,  u_n) -\tilde f}_{(V\times Q)'} &= \sup_{(\vv,q)\neq {\bf 0} } \frac{| a((\m_n,  u_n), (\vv,q))-\tilde f(\vv,q) |}{ \norm{(\vv,q)}_{V\times Q} }\\
&=\sup_{(\vv,q)\neq {\bf 0} } \frac{| a(\m_n,\vv) -b(\vv,   u_n)+b(\m_n,q) - \tilde f(\vv,q) |}{ \norm{(\vv,q)}_{V\times Q} }.
\end{split}
\eeq 
Noting from \eqref{reg-prob} that
\begin{align*}
\left|a(\m_n,\vv) -b(\vv,   u_n)+b(\m_n,q) - \tilde f(\vv,q)\right|
&=\frac 1 n \left|\intd{ |\nabla\cdot\m_n|^{ r^*-2} \nabla\cdot\m_n ,\nabla\cdot\vv}+ \intd{  u_n^\lambda,q} \right|\\
&\le \frac C n \left(\norm{\nabla\cdot\m_n}_{L^{ r^*}}^{{ r^* - 1}}\norm{\nabla\cdot\vv}_{L^{ r^*}} + \norm{  u_n}_{L^{ r}}\norm{q}_{L^{ r}} \right)\\
&\le \frac C n \left(\norm{\nabla\cdot\m_n}_{L^{ r^*}}^{{ r^* - 1}} + \norm{  u_n}_{Q} \right)\norm{(\vv,q)}_{V\times Q} .
\end{align*}
Hence 
\beq
\begin{split}
\norm{\mathcal A(\m_n,  u_n) -\tilde f}_{(V\times Q)'} 
\le \frac{C}{ n}\left( \norm{\nabla\cdot\m_n}_{L^{ r^*}}^{{ r^* - 1}} +\norm{  u_n}_Q  \right)\overset{n\to\infty}{\longrightarrow} 0.
\end{split}
\eeq 

The sequence $\mathcal A(\m_n,  u_n)$ converges strongly in $(V\times Q)'$ to $\tilde f$.  Thus we can
conclude that $\mathcal A(\m, u) = \tilde f$ in $(V\times Q)' $ (see e.g. \cite{z90}, p. 474), i.e., $(\m, u)$ is a solution of problem~\eqref{equivform}.

{\bf Step 4.} To show the uniqueness we consider two solutions $(\m_1,   u_1)$ and  $(\m_2,  u_2)$ of \eqref{WeakStationaryProb}. Using the test function $(\vv,q)=(\m_1-\m_2,  u_1-  u_2)$, we obtain 
\beq
\begin{aligned}
a(\m_1,\m_1-\m_2)-a(\m_2,\m_1-\m_2)  -\left( b(\m_1-\m_2,  u_1)- b(\m_1-\m_2,  u_2)  \right) = 0, \\
b(\m_1,   u_1-  u_2)-b(\m_2,   u_1-  u_2)=0.
\end{aligned}
\eeq
Adding these equations and using the monotonicity of $F(\cdot)$ in \eqref{F-mono} yield 
\beqs
\begin{split}
0=a(\m_1,\m_1-\m_2)-a(\m_2,\m_1-\m_2) 
\ge C_2 \left(\norm{ \m_1-\m_2}_{L^2}^2+\norm{ \m_1-\m_2}_{L^s}^s  \right) .
\end{split} 
\eeqs 
It follows that $\m_1=\m_2$. If $\m\in V$ is given then $u\in L^{ r}(\Omega)$ is defined as a solution of the variational equation $b(\vv, u)= \langle   u_b, \vv\cdot \nu  \rangle+a(\m,\vv)   $ for all $\vv\in V$. The uniqueness of $u$ is directly consequence of  Lemma~\ref{dualnorm}.
\end{proof}

\section {The semi-discrete problem}\label{SemiProb}

We return to the transient problem governed by \eqref{main-sys}. We discretize \eqref{main-sys} in time using the implicit Euler method. This yields not only a method to solve the transient problem numerically, but also an approach to prove its solvability, the technique of semi-discretization. We define a partition $0 = t_0 < t_1 < . . . < t_J = T$ of the segment $[0, T]$ into $J$ intervals of constant length $h = T/J$, i.e., $t_j = jh$ for $j = 0, \ldots ,J$. In the following for $j = 0, \ldots ,J$ we use the denotations
$ u^j := u(\cdot, jt)$ and $\m^j := \m(\cdot, jt)$ for the unknown solutions and, analogously defined, $  u_b^j$ for the boundary conditions and $f^j$ for the source term. 
\beq\label{semidiscreteProb}
\begin{aligned}
\Big(\sum_{i=0}^N a_i^j|\m^j|^{\alpha_i}\Big) \m^j = -\nabla  u^j  \quad x\in\Omega, \\
\phi\frac{ (u^j)^\lambda -  (u^{j-1})^\lambda}{h} +\nablad \m^j=f^j \quad x\in\Omega,\\
u = - u_b^j  \quad x\in\partial\Omega,\\
u(x,0)=u_0(x) \quad x\in\Omega.
\end{aligned}
\eeq
For each $j\in \{1,\ldots, J\}$, we will make the following assumptions: 
\beqs 
0<\underline{\phi}\le \phi(x) \le \overline\phi<\infty;  f^j \in L^{ r^*}(\Omega);  u_b^j\in W^{1/s, s}  (\partial \Omega),  u_0\in W^{1,s^*}(\Omega)\cap L^{ r}(\Omega); a_i^j(x)\in L^\infty(\Omega), i=0,\ldots,N.
\eeqs

{\bf Mixed formulation of the semi-discrete problem.} 
The discretization in time of the continuity equation \eqref{semidiscreteProb}  with the implicit Euler method yields for each $j \in \{1, . . . ,J\}$. Find 
  $(\m^j, u^j)  \in V\times Q$ such that
\beq\label{w1}
\begin{aligned}
\displaystyle \intd{ \big(\sum_{i=0}^N a_i^j|\m^j|^{\alpha_i}\big)\m^j ,\vv } - \intd{  u^j,\nabla\cdot \vv} =-\langle   u^j_b, \vv\cdot \nu  \rangle   \quad \text{ for all } \vv\in V,\\
\displaystyle \intd{ \phi \frac{(u^j)^\lambda }{h},q}+\intd{\nabla \cdot\m^j, q} =  \intd{ f^j, q }+\intd{ \phi\frac{  (u^{j-1})^{\lambda}}{h},q }  \quad \text{ for all } q\in Q,
\end{aligned}
\eeq
with $u^0 = u_0(x).$
Using $a$ and $b$ are defined in Section~\ref{StatProb}, we rewrite the mixed formulation \eqref{w1} in the following way: Find $(\m^j,  u^j)\in V \times Q$, such that
\beq\label{eqatstepj}
\begin{aligned}
a(\m^j, \vv) - b(\vv,  u^j) = -\langle   u^j_b, \vv\cdot \nu  \rangle \quad \text{ for all } \vv\in V,\\
\intd{ \phi\frac{ (u^j)^\lambda }{h},q} + b(\m^j,q) =\intd { \bar f^j ,q } \quad\text{ for all } q\in Q, 
\end{aligned}
\eeq
where 
$\bar f^j = f^j +\frac{\phi}{h}  (u^{j-1})^{\lambda}.$

The remainder of this section we restrict our considerations to the problem \eqref{eqatstepj} for a fixed time step $j$. For simplicity, we omit the superscript $j$.
\subsection{Regularization of the semi-discrete problem} 
We use the technique of regularization again.  For the fixed $\varep >$ 0, we consider the following regularized problem. Find $(\m_\varep,  u_\varep) \in V \times Q $ such that
\beq\label{reg-eqatstepj}
\begin{aligned}
a(\m_\varep, \vv)+\varep(|\nabla\cdot \m_\varep|^{ r^*-2} \nabla\cdot \m_\varep ,\nabla\cdot \vv) - b(\vv,  u_\varep) = -\langle   u_b, \vv\cdot \nu  \rangle&\quad  \text{ for all } \vv\in V,\\
\intd{ \frac{\phi}{h} u_\varep^\lambda,q}+ b(\m_\varep,q) =\intd { \bar f ,q }&\quad \text{ for all } q\in Q.
\end{aligned}
\eeq
The following result may be proved in much the same manner as Lemma~\ref{StatSol}. 
\begin{lemma}
For every $\varep$, there exists a unique solution $(\m_\varep,  u_\varep) \in V \times Q$ of the regularized semidiscrete problem \eqref{reg-eqatstepj}.
\end{lemma}

Next, we show that the solution $(\m_\varep,  u_\varep)$ of \eqref{reg-eqatstepj} is bounded independently of $\varep$.

\begin{lemma}\label{boundedness-discrete-sol}
There exists $\mathcal C>0$ independent of $\varep$ such that for sufficiently small $\varep>0$ the solution $(\m_\varep, u_\varep)$ of \eqref{reg-eqatstepj} satisfies 
\beq\label{discrete-bound}
\norm{ u_\varep}_Q+ \norm{\m_\varep}_V\le \mathcal C .
\eeq
\end{lemma}
\begin{proof} 
As in the proof of Lemma~\ref{stationary-sol-indep-eps}, we begin with an estimate for the norm of $\nabla\cdot\m_\varep$.
Using the second equation of \eqref{reg-eqatstepj} with $q= {\rm sgn}(\nabla\cdot \m_\varep) |\nabla\cdot \m_\varep|^{{ r^* - 1}}$, we obtain
\beq\label{divm-eps}
\norm{\nabla\cdot \m_\varep}_{L^{ r^*}} \le  \norm{\bar f}_{L^{ r^*}} + \frac{\bar\phi}{h} \norm{ u_\varep}_Q^\lambda.
\eeq
The estimation of $\norm{\m_\varep}_{L^s}$ is based on choosing the test function $(\vv,q)=(\m_\varep,  u_\varep)$ in \eqref{reg-eqatstepj}. Then we obtain the estimate 
\begin{multline}\label{asd}
a(\m_\varep, \m_\varep)+\varep\intd{|\nablad\m_\varep|^{ r^*-2} \nablad\m_\varep,\nablad\m_\varep}+\intd{ \phi\frac{u_\varep^\lambda }{h}, u_\varep}
\le \norm{ u_b}_{V'} \left(\norm{\m_\varep}_{L^s} + \norm{\nabla \cdot \m_\varep}_{L^{ r^*}} \right)  +\norm{\bar f}_{L^{ r^*}}\norm{ u_\varep}_Q.
\end{multline}
Thanks to the monotonicity of the function $F(\cdot)$ and \eqref{divm-eps}, it follows from \eqref{asd}  that
\beqs
C(\norm{\m_\varep}_{L^s}^{s}+ \norm{\m_\varep}^{2})+ \varep \norm{ \nabla\cdot \m_\varep}^{ r^*}_{L^{ r^*}} +\frac{\underline \phi}{h}\norm{ u_\varep}_Q^{ r} 
\le  \norm{ u_b}_{V'} \left(\norm{\m_\varep}_{L^s} + \norm{\bar f}_{L^{ r^*}} + \frac{\bar\phi}{h} \norm{ u_\varep}_Q^\lambda\right)  +\norm{\bar f}_{L^{ r^*}}\norm{ u_\varep}_{Q}. 
\eeqs
This and Young's inequality show that 
 \beqs
\norm{\m_\varep}_{L^s}^{s}+\frac{\underline \phi}{2h}\norm{ u_\varep}_Q^{ r}
\le  C \left(\norm{ u_b}_{V'}^{s^*} + \norm{ u_b}_{V'}\norm{\bar f}_{L^{ r^*}}+ \norm{ u_b}_{V'}^{ r} +\norm{\bar f}_{L^{ r^*}}^{ r^*}\right), 
\eeqs
which leads to  
\beq
\label{m-eps}
\norm{\m_\varep}_{L^s}\le C \left( \norm{ u_b}_{V'}^{s^*} + \norm{ u_b}_{V'}^{ r}  +\norm{\bar f}_{L^{ r^*}}^{ r^*}\right)^{1/s}, 
\quad \norm{ u_\varep}_Q\le C\left( \norm{ u_b}_{V'}^{s^*} + \norm{ u_b}_{V'}^{ r}  +\norm{\bar f}_{L^{ r^*}}^{ r^*}\right)^{1/ r}.
\eeq
Substituting \eqref{m-eps} into \eqref{divm-eps} we can assert that
\beq\label{new-divm-eps}
\begin{aligned}
\norm{\nabla\cdot \m_\varep}_{L^{ r^*}} 
&\le \norm{\bar f}_{L^{ r^*}}+ C\left( \norm{ u_b}_{V'}^{s^*} + \norm{ u_b}_{V'}^{ r}  +\norm{\bar f}_{L^{ r^*}}^{ r^*}\right)^{\lambda/ r} \\
&\le C \left(1+ \norm{ u_b}_{V'}^{s^*/ r^*} + \norm{ u_b}_{V'}^\lambda  +\norm{\bar f}_{L^{ r^*}}\right).
\end{aligned}
\eeq 
The assertion \eqref{discrete-bound} follows directly from \eqref{m-eps}--\eqref{new-divm-eps}.  
\end{proof}

\subsection{Solvability of the semi-discrete problem}
In the same manner as in Section 1, we pass the limit $\varep\to 0$ and obtain
the existence of a solution of the semi-discrete problem \eqref{w1}. 
\begin{theorem}\label{Sol-semidiscreteProb}
The mixed formulation \eqref{w1} of the semi-discrete problem \eqref{semidiscreteProb} possesses a unique solution $(\m, u)\in W(\rm{div};\Omega) \times L^{ r}(\Omega)$.
\end{theorem}
\begin{proof}
Analysis similar to that in the proof of Theorem \ref{SolofStationaryProb}, we add two equations in \eqref{eqatstepj} and obtain the nonlinear form $a$, defined on $(V \times Q) \times (V \times Q)'$, and the linear form $\tilde f \in (V \times Q)'$, defined by
\beqs
a( (\m, u),(\vv,q) )\eqdef a(\m,\vv)- b(\vv, u)+\intd{ \frac{\phi u }{h},q} +b(\m,q),\quad
\tilde f(\vv,q) \eqdef - \intb{ u_b,\vv\cdot\nu} +\intd{\bar f,q}. 
\eeqs
Again, the operator $\mathcal A : V \times Q \to (V \times Q)'$ is defined by $$\intb{\mathcal A(\uu, p), (\vv, q)}_{(V \times Q)'\times (V \times Q)} = a((\uu, p), (\vv, q)).$$ Choosing $\varep = 1/n$, we obtain a sequence of unique solutions $(\m_n,   u_n)$ of the regularized problems \eqref{reg-eqatstepj}. Owing to Lemma~\ref{boundedness-discrete-sol} the sequence $((\m_n,   u_n))_{n\in \N}$ is bounded in $V \times Q$. Hence there is a weakly convergent
subsequence, again denoted by $((\m_n,   u_n))_{n\in \N}$, which converges to $(\m, \rho) \in V \times Q$. In the same manner as in the proof of Theorem~\ref{SolofStationaryProb} we obtain the identity $\mathcal A(\m, \rho) = \tilde f$ in $(V \times Q)'$, i.e., $(\m, \rho)$ is a solution of the semi-discrete mixed formulation \eqref{w1}.
 
To show the uniqueness, we consider two solutions $(\m_1,   u_1)$ and  $(\m_2,  u_2)$ of \eqref{eqatstepj}. Using the test function $(\vv,q)=(\m_1-\m_2, u_1-  u_2)$, we obtain 
\begin{align*}
a(\m_1,\m_1-\m_2)-a(\m_2,\m_1-\m_2)  - b(\m_1-\m_2,  u_1-  u_2)= 0, \\
\intd{\phi\frac{   u_1^\lambda -  u_2^\lambda}{h},   u_1-  u_2} + b(\m_1-\m_2,   u_1-  u_2)=0.
\end{align*}
Adding the two equations then using \eqref{F-mono} and \eqref{Mono}   yields    
\beqs
\begin{split}
0&=a(\m_1,\m_1-\m_2)-a(\m_2,\m_1-\m_2) + \intd{\phi \frac{u_1^\lambda -  u_2^\lambda}{h},   u_1-  u_2}\\
&\ge C_2 \left(\norm{ \m_1-\m_2}^2+\norm{ \m_1-\m_2}_{L^s}^s  \right) + \int_\Omega \frac{\phi}{h} \frac{ |u_1-  u_2|^2}{|u_1|^{1-\lambda}  + |u_2|^{1-\lambda} }dx,
\end{split} 
\eeqs 
which proves $\m_1=\m_2$ and $  u_1=  u_2$ a.e. 
\end{proof}

\section {The transient problem}\label{TransProb}
We address the continuous transient problem. Due to the lack of regularity of the solution $\m$, it is impossible to handle more general boundary conditions as in the previous sections.  We will restrict our considerations here to the case of homogeneous Dirichlet boundary conditions
\beq\label{TProb}
\begin{aligned}
\left(\sum_{i=0}^N a_i(x,t)|\m(x,t)|^{\alpha_i}\right) \m(x,t) = -\nabla   u(x,t) & & \quad (x,t)\in\Omega\times (0,T) , \\
\phi(x) \partial_t u^\lambda(x,t) + \nablad\m(x,t)=f(x,t) & & \quad(x,t)\in\Omega\times (0,T),\\
  u(x,t) =0 & & \quad (x,t)\in\partial\Omega \times (0,T) ,\\
u(x,0)=u_0(x) & & \quad x\in\Omega.
\end{aligned}
\eeq
From now on the following assumptions will be needed
 \begin{itemize}
 \item[(H1)] $0<\underline{\phi}\le \phi(x) \le \overline\phi<\infty;
 f \in L^\infty (0,T;L^{ r^*}(\Omega) );
u_0\in W_0^{1,s^*}(\Omega)\cap L^{ r}(\Omega);  
a_i(\cdot,t)\in L^\infty(\Omega), i=0,\ldots,N.
$
\item [(H2)] The coefficient functions and $\|f\|$ to be Lipschitz continuous in time, i.e., there exists a constant $L$ such that, for every $0 \le t_1 \le t_2 \le T,$
\beqs
\norm{a_i(t_1)-a_i(t_2)}_{L^\infty}\le L |t_1-t_2|, \quad \text { and }\quad \norm{f(t_1)-f(t_2)}_{L^{ r^*}}\le L|t_1-t_2|. 
\eeqs

\item[(H3)] The degree of Forchheimer polynomial $F$ satisfies $\alpha_N={\rm deg} (F)\le \gamma $. It equivalents to  $ r\le s^*$.

\end{itemize}

\subsection{A priori estimates for the solutions of the semi-discrete problems}
As mentioned above we use the technique of semi-discretization in time  (see in \cite{raviart69})  to show the existence of solutions of the transient problem~\eqref{TProb}. The existence and uniqueness of the solutions to the semi-discrete problems has been established in Section~\ref{SemiProb}. In the next step, we consider the limit $h \to 0$. Similar to the regularized technique employed in the last two sections, we derive a priori estimates for the solutions of the semi-discrete problems, which are independent of $h.$ 

We investigate the semi-discrete problem~\eqref{w1} for homogeneous Dirichlet boundary condition. In this case problem~\eqref{w1} can read as the follow. Find $(\m^j,  u^j)\in W({\rm div}, \Omega) \times L^{ r}(\Omega) \equiv V\times Q$, such that
\beq\label{semidiscrete-prob}
\begin{aligned}
a(\m^j, \vv) - b(\vv,  u^j) = 0 \quad\text{ for all } \vv\in V,\\
\intd{\phi\frac{ (u^j)^\lambda -  (u^{j-1})^{\lambda} }{h}, q} + b(\m^j,q) =\intd{f^j, q}  \quad \text{ for all } q\in Q. 
\end{aligned}
\eeq
\begin{lemma}\label{boundedness1} For sufficiently small $h< 2^{-1} \underline{\phi}\lambda$, there exists $\mathcal C>0$ independent of $h$ and $J$ such that
\beq\label{indep1}
\norm{ u^j}_{L^{ r}} + \norm{ (u^j)^\lambda}_{L^{ r^*}} +   \norm{\m^j}+\norm{\m^j}_{L^s}\le \mathcal C\quad  \text{ for all }    j =1,2,\ldots, J.
\eeq
 \end{lemma}
\begin{proof}
Choosing $(\vv,q) = (\m^j, u^j)$ in \eqref{semidiscrete-prob} and adding the resulting equations yields
\beq\label{sum2eqs}
a(\m^j,\m^j)+\intd{\phi\frac{ (u^j)^\lambda -  (u^{j-1})^\lambda}{h},  u^j} =\intd{ f^j,  u^j}.
\eeq
From \eqref{ineqa} we have   
\begin{align*}
\intd{\phi\frac{ (u^j)^\lambda -  (u^{j-1})^\lambda }{h},  u^j}
\ge\frac {1}{\Dt r^*}\intd{ \phi,  (u^j)^{ r}- (u^{j-1})^{ r}}.  
\end{align*}
Due to the fact that  $a(\m^j,\m^j)\ge 0 $
and
\begin{align*}
\intd{f^j,  u^j}
  \le \frac {1}{ r^*}\norm{f^j}_{L^{ r^*}}^{ r^*} + \frac {1}{ r}\norm{ u^j}_{L^{ r}}^{ r}\le \frac {1}{ r^*}\norm{f^j}_{L^{ r^*}}^{ r^*} + \frac {1}{\underline{\phi} r}\intd{\phi, (u^j)^{ r}}, 
\end{align*}
it may be concluded that
\beqs
\intd{ \phi,  (u^j)^{ r}} -\intd{\phi, (u^{j-1})^{ r}}
 \le \Dt\norm{f^j}_{L^{ r^*}}^{ r^*} + \frac {\Dt  r^*}{\underline{\phi} r}\intd{\phi, (u^j)^{ r}}. 
\eeqs
If $\Dt$ sufficient small so that $\ell\Dt=\frac{\Dt  r^*}{\underline{\phi} r}=\frac{\Dt }{\underline{\phi}\lambda}<1$, which gives $\Dt<  \underline{\phi}\lambda$,   then 
\beqs
\intd{ \phi,  (u^j)^{ r}} \le \frac{1}{1-\ell\Dt}\left(\intd{\phi, (u^{j-1})^{ r}}+ \Dt\norm{f^j}_{L^{ r^*}}^{ r^*}\right).
\eeqs
By induction we find that
\beqs
\begin{split}
\intd{ \phi,  (u^j)^{ r}} 
&\le (1-\ell\Dt)^{-j}\Big(\intd{\phi, (u^0)^{ r}}+ \sum_{i=1}^{j}(1-\ell\Dt)^{-i+1} \Dt\norm{f^i}_{L^{ r^*}}^{ r^*}\Big).
\end{split}
\eeqs
Note that  $ (1- \ell\Dt)^{-j}\le e^{\frac{\ell T}{1-\ell\Dt}}< e^{2\ell T} $ for all  $\Dt<1/(2\ell) = \frac 1 2 \underline{\phi} \lambda$, it follows from above inequality that
\beq
\underline{\phi} \norm{u^j}_{L^{ r}}^{ r} \le \intd{ \phi,  (u^j)^{ r}}  \le  e^{2\ell T} 
\Big( \overline{\phi}\norm{u^0}_{L^{ r}}^{ r} + T\norm{f}_{L^\infty(0,T;L^{ r^*}) }^{ r^*} \Big).
\eeq
This leads to  
\beq\label{Buj}
\norm{u^j}_{L^{ r}}\le C_1, 
\eeq 
where $C_1=\left(\underline{\phi}^{-1}e^{2\ell T} 
\big( \overline{\phi}\norm{u^0}_{L^{ r}}^{ r} + T\norm{f}_{L^\infty(0,T;L^{ r^*}) }^{ r^*} \big) \right)^{1/ r}$. 
It follows easily that
\beq\label{Bujlambda}
\norm{(u^j)^\lambda}_{L^{ r^*}} =\norm{u^j}_{L^ r}^\lambda \le C_1^{\lambda}. 
\eeq

Using the test function $q = u^j- u^{j-1}$,  we obtain the from the second equation in  \eqref{semidiscrete-prob} that 
\beq\label{fEq}
\intd{\phi\frac{ (u^j)^\lambda -  (u^{j-1})^\lambda}{h},  u^j -  u^{j-1}}+ b(\m^j,u^j -  u^{j-1} ) =\intd{ f^j,  u^j - u^{j-1}}.
\eeq
Now taking $\vv=\m^j$  at time step $j$ and $j-1$ from the first equation in \eqref{semidiscrete-prob}, we have  
\beqs 
a(\m^j, \m^j) - b(\m^j,  u^j) = 0,\quad  \text{ and }  \quad a(\m^{j-1}, \m^j) - b(\m^j,u^{j-1}) =0, 
\eeqs
which implies that
\beq\label{sEq}
a(\m^j, \m^j)- a(\m^{j-1}, \m^j)= b(\m^j,  u^j-   u^{j-1}).
\eeq
Combining \eqref{fEq} and \eqref{sEq} shows that  
\beqs
\intd{\phi\frac{ (u^j)^\lambda -  (u^{j-1})^\lambda}{h},  u^j -  u^{j-1}} +a(\m^j, \m^j)- a(\m^{j-1}, \m^j) = \intd{ f^j,  u^j - u^{j-1}}.
\eeqs
Summing up this equation for $k=1,2,\ldots, j$ yields   
\beq\label{neweq}
  \sum_{k=1}^j \intd{\phi\frac{ (u^k)^\lambda -  (u^{k-1})^\lambda}{h},  u^k -  u^{k-1}}+ a(\m^k, \m^k)-a(\m^{k-1}, \m^k) = \sum_{k=1}^j \intd{f^k  ,u^k-   u^{k-1}}.
\eeq
 We will estimate \eqref{neweq} term by term.
 
The last term on the right hand side of \eqref{neweq} are bounded by using H\"older's inequality and  \eqref{indep1}   
\beq\label{a2}
\begin{aligned}
\sum_{k=1}^j  \intd{f^k , u^k-   u^{k-1}} &= \intd{f^j , u^j}- \intd{f^1,u^0}+\sum_{k=1}^{j-1} \intd{f^k-f^{k+1}, u^k}\\
&\le \norm{f^j}_{L^{ r^*}}\norm{ u^j}_{L^{ r}} + \norm{f^1}_{L^{ r^*}}\norm{u^0}_{L^{ r}} +\sum_{k=1}^{j-1}\norm{f^k-f^{k+1}}_{L^{ r^*}} \norm{ u^k}_{L^{ r}}\\
&\le  \norm{f}_{L^\infty(0,T;L^{ r^*})}(C_1 +\norm{u^0}_{L^{ r}}) +C_1 L T
\le  (C_1+1) \left( \norm{f}_{L^\infty(0,T;L^{ r^*})}+ \norm{u^0}_{L^{ r}}+LT\right).  
\end{aligned}
\eeq
For the last two terms on the left hand side of \eqref{neweq}, we rewrite as 
\beq\label{twolastsum}
 \sum_{k=1}^j a(\m^{k}, \m^k)-a(\m^{k-1}, \m^k)
  = \sum_{k=1}^j  \sum_{i=0}^N \int_\Omega\left( a_i^k|\m^k|^{\alpha_i+2}- a_i^{k-1}|\m^{k-1}|^{\alpha_i}\m^{k-1} \cdot\m^k \right)dx. 
\eeq
By Young's inequality
\beq\label{intsum}
\begin{split}
\sum_{i=0}^N a_i^{k-1}|\m^{k-1}|^{\alpha_i}\m^{k-1} \cdot\m^k 
&\le \sum_{i=0}^N a_i^{k-1}\left( \frac{\alpha_i+1}{\alpha_i+2}|\m^{k-1}|^{\alpha_i+2}+\frac{1}{\alpha_i+2} |\m^k|^{\alpha_i+2}\right)\\
&=\sum_{i=0}^N \left(\frac{(\alpha_i+1)a_i^{k-1}}{\alpha_i+2} |\m^{j-1}|^{\alpha_i+2}+ \frac{a_i^{k-1} -a_i^k }{\alpha_i+2}|\m^k|^{\alpha_i+2}+ \frac{a_i^k}{\alpha_i+2}  |\m^k|^{\alpha_i+2}\right).
\end{split}
\eeq
Substituting \eqref{intsum} into \eqref{twolastsum} yields  

\begin{multline}\label{a3}
 \sum_{k=1}^j a(\m^k, \m^k)-a(\m^{k-1}, \m^k)
 \ge \sum_{k=1}^j \sum_{i=0}^N \int_\Omega \Big(\frac{\alpha_i+1}{\alpha_i+2} ( a_i^k|\m^k|^{\alpha_i+2}-a_i^{k-1}|\m^{k-1}|^{\alpha_i+2}) +  \frac{a_i^k-a_i^{k-1}}{\alpha_i+2}|\m^k|^{\alpha_i+2}\Big) dx\\
=  \sum_{i=0}^N \int_\Omega \frac{\alpha_i+1}{\alpha_i+2} (a_i^0 |\m^0|^{\alpha_i+2} + a_i^j|\m^j|^{\alpha_i+2})  dx+ \sum_{i=0}^N\sum_{k=1}^j  \int_\Omega\frac{a_i^k-a_i^{k-1} }{\alpha_i+2}|\m^k|^{\alpha_i+2} dx\\
\ge \int_\Omega \left( \frac{1}{2}a_0^j |\m^j|^2 + \frac{s-1}{s} a_N^j|\m^j|^{s}\right) dx+ \sum_{i=0}^N\sum_{k=1}^j  \int_\Omega\frac{a_i^k-a_i^{k-1} }{\alpha_i+2}|\m^k|^{\alpha_i+2} dx.
\end{multline}
Due to \eqref{Mono}, the first term 
\beq\label{a4}
\sum_{k=1}^j \intd{\phi\frac{ (u^k)^\lambda -  (u^{k-1})^\lambda}{h},  u^k -  u^{k-1}}\ge 0. 
\eeq
Substituting  \eqref{a2}, \eqref{a3} and \eqref{a4} into  \eqref{neweq} yields  
\begin{multline}
\frac{1}{2}a_0^j  \norm{\m^j}_{L^2}^2 + \frac{s-1}{s}a_N^j \norm{\m^j}_{L^s}^{s}  
\le - \sum_{i=0}^N\sum_{k=1}^j  \int_\Omega\frac{a_i^k-a_i^{k-1} }{\alpha_i+2}|\m^k|^{\alpha_i+2} dx +(C_1+1) \left( \norm{f}_{L^\infty(0,T;L^{ r^*})}+ \norm{u^0}_{L^{ r}}+LT\right)\\
\le \frac{L}{2\underline{a}} \sum_{i=0}^N\sum_{k=1}^j  \int_\Omega \Dt a_i^k|\m^k|^{\alpha_i+2} dx+ (C_1+1) \left( \norm{f}_{L^\infty(0,T;L^{ r^*})}+ \norm{u^0}_{L^{ r}}+LT\right)\\
\le \frac{L}{2\underline{a}} \sum_{k=1}^j  \Dt a(\m^k, \m^k) + (C_1+1) \left( \norm{f}_{L^\infty(0,T;L^{ r^*})}+ \norm{u^0}_{L^{ r}}+LT\right).
\end{multline}
On the other hand by \eqref{sum2eqs}, 
\begin{multline*}
 \sum_{k=1}^j  \Dt a(\m^k, \m^k)=  -\sum_{k=1}^j  \intd{\phi( (u^k)^\lambda -  (u^{k-1})^\lambda),  u^k} +\Dt \intd{ f^k,  u^k}\\
\le  \frac {1}{r} \sum_{k=1}^j  \intd{\phi, (u^{k-1})^r -  (u^k)^r} +\Dt \norm{f^k}_{r^*} \norm{ u^k}_{L^r}
\le \frac {1}{r}\intd{\phi, (u^0)^r -  (u^j)^r} +  C_1 \sum_{k=1}^j   \Dt \norm{f^k}_{r^*}  \\
\le \frac {1}{r}\bar{\phi} (\norm {u^0}_{L^r}^r +\norm{u^j}_{L^r}^r) + C_1 T \norm{f}_{L^\infty(0,T; L^{r^*})}
\le \Big(\frac{\bar{\phi}}{r}+1\Big) (C_1^r+1)\left(\norm {u^0}_{L^r}^r+ T \norm{f}_{L^\infty(0,T; L^{r^*})}+1\right). 
\end{multline*}
Consequently,
\beq\label{direct}
\begin{split}
\frac{1}{2} a_0^j \norm{\m^j}_{L^2}^2 + \frac{s-1}{s} a_N^j\norm{\m^j}_{L^s}^{s}  
&\le \frac{L}{2\underline{a}}\Big(\frac{\bar{\phi}}{r}+1\Big)  (C_1^r+1)\left(\norm {u^0}_{L^r}^r+ T \norm{f}_{L^\infty(0,T; L^{r^*})}\right) \\
&\quad +(C_1+1) \left( \norm{f}_{L^\infty(0,T;L^{ r^*})}+ \norm{u^0}_{L^{ r}}+LT\right).
\end{split}
\eeq

The assertion \eqref{indep1} directly follows from \eqref{Buj}, \eqref{Bujlambda} and \eqref{direct}.    
\end{proof}

\begin{lemma}\label{boundedness2} For sufficiently small $h$, there exists $\mathscr C>0$  independent of $h$ and $J$ such that
\beq\label{indep-derv}
\sum_{j=1}^J h \int_\Omega 
\left|\frac{ u^j-  u^{j-1}}{\Dt}\right|^{ r} dx  \le \mathscr C.   
\eeq
\end{lemma}
\begin{proof}
We rewrite \eqref{neweq} as the form 
\beq\label{a1}
\sum_{j=1}^J \int_\Omega 
\phi\frac{ (u^j)^\lambda -  (u^{j-1})^\lambda}{h}( u^j-   u^{j-1}) dx
 =  \sum_{j=1}^J a(\m^{j-1}, \m^j)-a(\m^j, \m^j) + \intd{f^j  ,u^j-   u^{j-1}}.
\eeq
From \eqref{a3} we have 
\begin{multline}\label{ab3}
 \sum_{j=1}^J a(\m^{j-1}, \m^j)-a(\m^j, \m^j)
 \le \sum_{j=1}^J \sum_{i=0}^N \int_\Omega \Big(\frac{\alpha_i+1}{\alpha_i+2} (a_i^{j-1}|\m^{j-1}|^{\alpha_i+2} - a_i^j|\m^j|^{\alpha_i+2}) +  \frac{a_i^{j-1} -a_i^j}{\alpha_i+2}|\m^j|^{\alpha_i+2}\Big) dx\\
\le  \sum_{i=0}^N \Big(\bar{a}(\norm{\m^0}_{L^s}^{\alpha_i+2} + \norm{\m^J}_{L^s}^{\alpha_i+2}) + \sum_{j=1}^J  \norm {a_i^{j-1} -a_i^j }_{L^\infty}\norm{\m^j}_{L^s}^{\alpha_i+2}\Big)\\
\le  (\bar{a} +1) (N+1)\Big(1+ \norm{\m^0}_{L^s}^{\alpha_N+2} +\mathcal C^{\alpha_N+2} + LT(1+ \mathcal C^{\alpha_N+2})\Big).
\end{multline}
It follows from \eqref{a1}, \eqref{a2}, \eqref{ab3} and \eqref{indep1} that 
\beq\label{sumb}
\sum_{j=1}^J  \int_\Omega 
\frac{\phi \Dt }{ |u^j|^{1-\lambda}+|u^{j-1}|^{1-\lambda}} \left| \frac{ u^j-  u^{j-1}}{h}\right|^2 dx 
\le \sum_{j=1}^J \int_\Omega 
\phi\frac{ (u^j)^\lambda -  (u^{j-1})^\lambda}{h}( u^j-   u^{j-1}) dx  \le C_2,
\eeq
where
$$
C_2 \eqdef  C (C_1+1) \left( \norm{f}_{L^\infty(0,T;L^{ r^*})}+ \norm{u^0}_{L^{ r}}+LT\right)+ (\bar{a} +1) (N+1)\Big(1+ \norm{\m^0}_{L^s}^{\alpha_N+2} +\mathcal C^{\alpha_N+2} + LT(1+ \mathcal C^{\alpha_N+2})\Big).
$$

On the other hand, by H\"older's inequality
\beqs
\begin{split}
&\sum_{j=1}^J h \int_\Omega \left|\frac{ u^j-  u^{j-1}}{\Dt}\right|^{ r} dx\\
&\quad\le \sum_{j=1}^J  \left( \int_\Omega \Dt  \left|\frac{\phi }{ |u^j|^{1-\lambda}+|u^{j-1}|^{1-\lambda}}\right|^{- r/(2- r)} dx\right)^{1- r/2}
\left(\int_\Omega\frac{\phi \Dt}{ |u^j|^{1-\lambda}+|u^{j-1}|^{1-\lambda}} \left| \frac{ u^j-  u^{j-1}}{h}\right|^2 dx\right)^{r/2}\\
&\quad\le   \left(\sum_{j=1}^J\int_\Omega \Dt  \left|\frac{\phi }{ |u^j|^{1-\lambda}+|u^{j-1}|^{1-\lambda}}\right|^{- r/(2- r)} dx\right)^{1-r/2}
\left(\sum_{j=1}^J\int_\Omega\frac{\phi \Dt}{ |u^j|^{1-\lambda}+|u^{j-1}|^{1-\lambda}} \left| \frac{ u^j-  u^{j-1}}{h}\right|^2 dx\right)^{ r/2}. 
\end{split}
\eeqs
Since 
\begin{align*}
 \sum_{j=1}^J\int_\Omega \Dt  \left|\frac{\phi }{ |u^j|^{1-\lambda}+|u^{j-1}|^{1-\lambda}}\right|^{-r/(2- r)} dx
&\le  \sum_{j=1}^J\Dt \int_\Omega   \phi^{-r/(2- r)} \left(|u^j|^{1-\lambda}+|u^{j-1}|^{1-\lambda}\right)^{r/(2- r)} dx\\
&\le\sum_{j=1}^J\Dt  \underline{\phi}^{- r/(2- r)} 2^{2( r-1)/(2- r)}\Big( \norm{u^j}_{L^{ r}}^{ r}+\norm{u^{j-1}}_{L^{ r}}^{ r}\Big),
\end{align*}
we use \eqref{sumb} and  \eqref{indep1} to conclude that 
\beqs
\sum_{j=1}^J h \int_\Omega \left|\frac{ u^i-  u^{i-1}}{\Dt}\right|^{ r} dx \le (2C_2)^{ r/2} \mathcal C^{ r(1- r/2)} T^{1- r/2} \underline{\phi}^{- r/2}.
\eeqs
This completes the proof. 
\end{proof}

Next, we show that the mixed formulation \eqref{semidiscrete-prob} is equivalent to a variational formulation of the time-discretized parabolic equation.  To this end, we recall the nonlinear mapping $K$ of \eqref{rua}.  For fixed time $t=t_j$, we define the nonlinear mapping $K^j: \Omega \times \R^+ \to \R^+$ (see in \eqref{Kdef}) and its inverse defined by  
\beq\label{eq2j}
F^j(x,z)=a_0(x,t_j)z^{\alpha_0} + a_1(x,t_j)z^{\alpha_1}+\cdots +a_N(x,t_j)z^{\alpha_N},\quad z\ge 0.
\eeq 
\begin{lemma}\label{equivProbs}

(i)\quad If $ u^j\in R(\Omega)=\{r\in L^{ r}(\Omega),  r=0 \text { on } \partial\Omega,\nabla r\in (L^{s^*}(\Omega))^d \}$ is a solution of the variational formulation. Find $ u^j\in R(\Omega)$ such that
\beq\label{VarProb}
\intd{\phi\frac{ (u^j)^\lambda-  (u^{j-1})^\lambda }{h},  q} + \intd{ K^j(x,|\nabla  u^j|)\nabla  u^j, \nabla q } =\intd{ f^j,q} \quad \text{ for all } q\in R(\Omega) 
\eeq
 then $(-K^j(x,|\nabla  u^j|)\nabla  u^j,  u^j)  $ is a solution of the mixed formulation~\eqref{semidiscrete-prob}. 

(ii)\quad  If $(\m^j,  u^j)\in W(\rm{div},\Omega)\times L^{ r}(\Omega)$ is a solution of the mixed formulation~\eqref{semidiscrete-prob} then
$ u^j$ is a solution of the variational formulation \eqref{VarProb}.  In particular, $ u^j \in R(\Omega) $.
\end{lemma}
\begin{proof} 

(i) Let $ u^j$ be a solution of \eqref{VarProb}. We define $\m^j=-K^j(x,|\nabla  u^j|)\nabla  u^j$. Then Green's formula yields
\beqs
\intd{F^j(x,|\m^j|)\m^j , \vv }  =-\intd{\nabla  u^j, \vv} =\intd{ u^j, \nabla\cdot \vv}  \quad \text{ for all } \vv\in V.
\eeqs
This is the first equation in \eqref{semidiscrete-prob}. To derive the second equation in~\eqref{semidiscrete-prob}, we consider \eqref{VarProb} for
$q\in \mathcal D(\Omega)\subset R(\Omega)$ 
\beqs
\intd{\phi\frac{ (u^j)^\lambda-  (u^{j-1})^{\lambda}}{h}, q} -(\m^j, \nabla q)=  \intd{f^j, q }, 
\eeqs
and then apply Green's formula we obtain  
\beqs
\intd{\phi\frac{ (u^j)^\lambda-  (u^{j-1})^\lambda}{h}, q}+ \intd{ \nabla \cdot \m^j, q } =\intd{f^j, q}.   
\eeqs

Because  $\mathcal D(\Omega)$ is densely embedded into $L^{ r}(\Omega)$, the second equation in \eqref{semidiscrete-prob} follows.

(ii) Let $(\m^j, u^j)$ be the solution of \eqref{semidiscrete-prob}. Applying Green's formula implies
\beqs
\intd{ F^j(x,|\m^j|)\m^j,  \vv}= \intd{ \nabla \cdot \vv ,  u^j } = \intd{ -\nabla  u^j, \vv }\quad \text{ for all } \vv\in (\mathcal D(\Omega))^d.   
\eeqs
Thus in the sense of distributions it holds $\nabla  u^j= -F^j(x,|\m^j|)\m^j \in (L^{s*}(\Omega))^d $. 

Consequently, $ u^j\in \{r\in L^{ r}(\Omega), \nabla r\in (L^{s^*}(\Omega))^d \}$  and $\m^j=-K^j(x,|\nabla  u^j|)\nabla u^j$. 

To prove that $ u^j$ fulfills \eqref{VarProb}, we consider $q\in R(\Omega)\subset L^{ r}(\Omega) $ in the first equation of \eqref{semidiscrete-prob}. Using integration by parts, we have  
 \beqs
 \intd{f^j,q}= \intd{\phi\frac{ (u^j)^\lambda -  (u^{j-1})^\lambda}{h}, q } + \intd{\nabla\cdot \m^j, q } = \intd{\phi\frac{ (u^j)^\lambda -  (u^{j-1})^\lambda}{h}, q } + \intd{ K^j(x, |\nabla  u^j|)\nabla  u^j, \nabla q }.
 \eeqs
 Finally, we consider again the first equation of \eqref{semidiscrete-prob} for $\vv \in(\mathcal D(\bar \Omega))^d$. Using integration by parts, we obtain
 \beqs
 0=-\intd{F^j(x,|\m^j|)\m^j ,\vv}+ \intd{\nabla \cdot \vv , u^j } = \intd{ \nabla u^j ,\vv}+ \intd{\nabla \cdot \vv,  u^j} = \int_{\partial\Omega}  u^j (\vv\cdot\nu) d\sigma. 
 \eeqs
 Consequently, $ u^j=0$ on $\partial \Omega$, i.e., $ u^j\in R(\Omega)$.  
  \end{proof}
  
Using this equivalence, we obtain a bound for $ u^j$ in the norm of $R(\Omega)$ defined by \[\norm{r}_R = \norm{r}_{L^{ r}}+ \norm{\nabla r}_{L^{s^*}}.\]
\begin{lemma}\label{boundedness3} For sufficiently small $h$, there is a constants $\mathcal C>0$ independent of $h$ and $J$, such that
\begin{align}
\label{indep2} \norm{ u^j}_R+ \norm{\frac{ (u^j)^\lambda -  (u^{j-1})^\lambda }{h} }_{R'}+\norm{\nabla \cdot \m^j}_{R'} &\le \mathcal C\quad \text{for all } j=0,1,2,\ldots, J.
\end{align}

\end{lemma}
\begin{proof}
To verify $\norm{ u^j}_R$ is bounded, it is sufficient to use \eqref{indep1} together with the observation that 
\begin{align*}
\norm{\nabla  u^j}_{L^{s^*}}^{s^*} =\norm{F^j(x,|\m^j|)\m^j}_{L^{s^*}}^{s^*}& \le C \sum_{i=0}^N\int_\Omega |\m^j|^{(\alpha_i+1)s^*} dx\\
&\le C\sum_{i=0}^N\norm{\m^j}_{L^s}^{(\alpha_i+1)s^*} \le CN\left(\norm{\m^j}_{L^s}^{s^*}+\norm{\m^j}_{L^s}^{(\alpha_N+1)s^*}\right) .
\end{align*} 
We next to prove that $  \norm{\frac{ (u^j)^\lambda -  (u^{j-1})^\lambda }{h} }_{R'}$ is bounded.
 
By means of \eqref{VarProb}, we have for all $q\in R(\Omega)$,
\beqs
\begin{split}
\left |\intd{\phi\frac{ (u^j)^\lambda-  (u^{j-1})^\lambda}{h},q } \right |=\left| \intd{ f^j, q} -\intd{ K^j(x,|\nabla  u^j|)\nabla  u^j,q } \right|&=\left| \intd{f^j,q} + \intd{\m^j,\nabla q}\right|\\
  &\le \left(\norm{f^j}_{L^{ r^*}} + \norm{\m^j}_{L^s}\right)\norm{q}_R.
\end{split}
\eeqs
Thanks to the boundedness of the function $\phi$ and \eqref{indep1} we see that 
\beq\label{diffquotion}
\norm{\frac{ (u^j)^\lambda-  (u^{j-1})^\lambda}{h} }_{R'}\le C  \left(\norm{f^j}_{L^{ r^*}} + \norm{\m^j}_{L^s}\right) \le \mathcal C 
\eeq 
which is part of conclusion \eqref{indep2}. 
 
From the second equation of \eqref{semidiscrete-prob} yields
\beqs
\left|\intd{\nabla\cdot \m^j, q }\right| 
 = \left |\intd{ f^j, q }- \intd{\phi\frac{ (u^j)^\lambda-  (u^{j-1})^\lambda}{h}, q } \right | 
 \le \left(\norm{f^j}_{L^{ r^*}}+ \bar \phi \norm{\frac{ (u^j)^\lambda-  (u^{j-1})^\lambda}{h} }_{L^{ r^*}}   \right) \norm{q}_R,
\eeqs
which implies that 
\beqs
\norm{\nabla \cdot \m^j}_{R'} \le \norm{f^j}_{L^{ r^*}}+ \bar \phi \norm{\frac{ (u^j)^\lambda-  (u^{j-1})^\lambda}{h} }_{L^{ r^*}} \le \norm{f^j}_{L^{ r^*}}+ \bar \phi \norm{\frac{ (u^j)^\lambda-  (u^{j-1})^\lambda}{h} }_{R'}\le \mathcal C .
\eeqs
The proof is complete.    
\end{proof}

\subsection{Solvability of the continuous problem} 
Due to the existence of unique solutions to the semi-discrete mixed formulation \eqref{semidiscrete-prob}, we obtain for every $J\in \N$ a $J+1$-tuple of solutions $(\m^j,  u^j)_{j=0,\ldots,J}\in ( W(\rm{div}, \Omega) \times L^{ r}(\Omega) ) ^{J+1}. $ We denote these $J + 1$-tuples with $\m_{\Dt}:=(\m^j)_{j=0,\ldots, J}\in  ( W(\rm{div}, \Omega) ) ^{J+1} $ and $  u_{\Dt}:= ( u^j)_{j=0,\ldots, J}\in  ( L^{ r}(\Omega) ) ^{J+1}$. We define step function by 
\beqs
\pi u_{\Dt}(t)=\sum_{j=1}^J \chi_j(t)u^j
 \in L^\infty(0,T; R(\Omega)),
\eeqs
where $\chi_j(t)$ is the characteristic function on the interval $(t_{j-1}, t_j]$ for $j=1,2,\ldots J$. We also defined a piecewise linear (in time) functions
\beqs 
\Pi u_{\Dt}(t)=\sum_{j=1}^J \chi_j(t)  \left(\frac{ u^j-  u^{j-1}}{\Dt} (t-t_j) +  u^j\right ).
\eeqs
In addition, we use piecewise constant approximations ${a_i}_{\Dt}$ and $f_{\Dt}$ of the coefficient functions $a_i$ and $f$, and piecewise constant operators $F_{\Dt}$ and $K_{\Dt}$. According to Lemmas~\ref{boundedness1} and \ref{boundedness3} the following bounds hold for sufficiently small $h$.
\begin{align*}
&\norm{\pi u_{\Dt}}_{L^\infty(0,T; R(\Omega))}\le \mathcal C ,  
&& \norm{\partial_t \Pi u_{\Dt} }_{L^ r(0,T;L^ r(\Omega)  )}\le \mathcal C, \\
&\norm{(\pi u_{\Dt})^\lambda}_{L^\infty(0,T; L^{ r^*}(\Omega))}\le \mathcal C,       
&&\norm{ \partial_t \Pi u_{\Dt}^\lambda }_{L^\infty(0,T;R'(\Omega))}\le \mathcal C, \\
&\norm{\pi\m_{\Dt}}_{L^\infty(0,T; (L^{s}(\Omega))^d)}\le \mathcal C,   
 &&\norm{\pi\nabla\cdot\m_{\Dt}}_{L^\infty(0,T;R'(\Omega))}\le \mathcal C,\\
&\norm{ F(x,|\pi\m_{\Dt}|)\pi\m_{\Dt}}_{L^\infty(0,T; (L^{s^*}(\Omega))^d)}\le \mathcal C, 
&&\norm{ u^J}_{L^ r(\Omega)}\le \mathcal C.
\end{align*}

Thus the exist a subsequences, again indexed by $\Dt$, that converge in corresponding weak*-topology; in detail 
 \begin{align}
\label{b1} &\pi u_{\Dt} \stackrel{*}{\rightharpoonup} u \, \text { in }\, L^\infty(0,T; R(\Omega)), 
&& \partial_t \Pi u_{\Dt} \stackrel{ *}{\rightharpoonup} u' \, \text { in }\, L^\infty(0,T;L^ r(\Omega))\\
\label{b4} &(\pi u_{\Dt})^\lambda \stackrel{*}{\rightharpoonup} U \, \text { in }\, L^\infty(0,T; L^{ r^*}(\Omega)), 
&& \partial_t \Pi u_{\Dt}^\lambda \stackrel{ *}{\rightharpoonup} U' \, \text { in }\, L^\infty(0,T;R'(\Omega))\\
\label{b2} & \pi\m_{\Dt}\stackrel{ *}{\rightharpoonup}  \m \, \text { in } \, L^\infty(0,T; (L^{ r}(\Omega))^d),
&&\pi\nabla\cdot\m_{\Dt} 	\stackrel{ *}{\rightharpoonup}  \bar\m \, \text { in } \, L^\infty(0,T;R'(\Omega))\\
\label{b3} &F(x,|\pi\m_{\Dt}|)\pi\m_{\Dt}	\stackrel{ *}{\rightharpoonup}  \hat F\, \text { in }\, L^\infty(0,T; (L^{s^*}(\Omega))^d), 
&& u^J\stackrel{ }{\rightharpoonup}  u_T \,\text { in }\quad L^{ r}(\Omega).
\end{align}

\begin{lemma}\label{identeq0}
(i)  The identity $ U=u^\lambda$ hold in the sense of distribution from $(0,T)$ to $L^{ r^*}(\Omega)$. 

(ii) The identity $u'=\partial_t u$  holds in the sense of distribution from $(0,T)$ to $L^{ r}(\Omega)$. That is  for all $\varphi\in \mathcal D ((0,T))$, 
\beqs
\int_0^T u'(t)\varphi (t) dt = - \int_0^T u(t)\varphi'(t) dt \, \text{ in } L^{ r}(\Omega).  
\eeqs

(iii) The identity $U'=\partial_t U$  holds in the sense of distribution from $(0,T)$ to $R'$. That is  for all $\varphi\in \mathcal D ((0,T))$, 
\beqs
\int_0^T U'(t)\varphi (t) dt = - \int_0^T U(t)\varphi'(t) dt \, \text{ in } R'(\Omega).  
\eeqs

(iv) The identity $\bar \m= \nabla\cdot \m$ hold in the sense of distribution on $\Omega$ hold for almost everywhere in   $(0,T)$. That is  for all $\psi\in \mathcal D (\Omega)$, 
\beq\label{coincide2}
\intd{\bar \m, \psi} = - \intd{\m, \nabla \psi } \, \, \text{ a.e in~}(0,T).  
\eeq

(v) The identity $\hat F = -\nabla u$ hold in $L^\infty (0,T; (L^{s^*}(\Omega))^d )$. 
That is  for all $\vv\in L^1(0,T; (L^{s}(\Omega))^d )$,
\beq\label{coincide3}
\int_0^T (\hat F, \vv) dt = - \int_0^T\intd{\nabla u,  \vv}  dt.  
\eeq
\end{lemma}
\begin{proof}
(i) Since $\pi u_{\Dt} \in L^\infty(0,T;R(\Omega))$, $\Pi u_{\Dt} \in L^\infty(0,T;R(\Omega))$. In particular,   $\Pi u_{\Dt} \in L^ r(0,T;L^ r(\Omega))$ and $\partial_{x_i} (\Pi u_{\Dt}) \in L^ r(0,T;L^ r(\Omega))$.   Due to  \eqref{indep-derv}, $\partial_t \Pi u_{\Dt} \in L^{ r}(0,T; L^ r(\Omega))$. This implies $\Pi u_{\Dt}\in W^{1, r}((0,T)\times \Omega  ) $. The Rellich- Kondrachov theorem yields that  $W^{1, r}((0,T)\times \Omega)$ is compactly embedded in $L^{ r}((0,T)\times \Omega  )$. There is the subsequence $\Pi u_h \to u$ strongly in $L^ r( (0,T)\times\Omega )$. Thus $\lim_{h\to 0}\Pi u_h^\lambda = u^\lambda$ a.e in  $L^{ r^*}((0,T)\times \Omega)$.
Since $\Pi u_{\Dt}^\lambda$ is bounded in $L^{ r^*}((0,T)\times \Omega)$ we conclude that $\Pi u_{\Dt}^\lambda$ converge weakly to $u^\lambda$ in $L^{ r^*}((0,T)\times \Omega)$ that is for all $\varphi\in L^{ r}((0,T),\Omega)$, 
\beqs
\lim_{\Dt\to 0} \int_0^T \intd{\Pi u_{\Dt}^\lambda ,\varphi }dt = \int_0^T \intd{u^\lambda ,\varphi }dt.
\eeqs 
On the other hand,
\beqs
\lim_{\Dt\to 0} \int_0^T \intd{\Pi u_{\Dt}^\lambda ,\varphi }dt=\lim_{\Dt\to 0} \int_0^T \intd{\pi u_{\Dt}^\lambda ,\varphi }dt  =\lim_{\Dt\to 0} \int_0^T \intd{(\pi u_{\Dt})^\lambda ,\varphi }dt  =\int_0^T\intd{ U,\varphi}dt. 
\eeqs
Then the assertion follows. 

(ii) Let $\varphi\in \mathcal D (0,T)$ then  
\beqs
\int_0^T u'(t) \varphi (t) dt = \lim_{\Dt\to 0} \int_0^T \partial_t \Pi u_{\Dt}  \varphi (t) dt=  -\lim_{\Dt\to 0}\int_0^T  \Pi u_{\Dt}  \varphi' (t) dt
=-\lim_{\Dt\to 0}\int_0^T  \pi u_{\Dt}  \varphi' (t) dt=-\int_0^T u \varphi' (t) dt.
\eeqs

(iii) Similar to (ii)

(iv) Let $\psi\in \mathcal D(\Omega)$ and $\varphi\in \mathcal D(0,T)$ then
\beqs
 \int_0^T \intd{\bar \m, \psi} \varphi (t) dt = \lim_{\Dt\to 0}\int_0^T \intd{\pi\nabla\cdot\m_{\Dt},\psi}\varphi(t) dt
  = - \lim_{\Dt\to 0}\int_0^T \intd{ \pi\m_{\Dt},\nabla \psi}\varphi(t) dt= - \int_0^T (\m,\nabla\psi) \varphi(t) dt.
\eeqs

(v) For all $\vv\in L^1(0,T,(L^s(\Omega))^d)$, 
\beqs
\int_0^T( \hat F, \vv)dt =\lim_{\Dt\to 0}\int_0^T \intd{  F(x,|\pi\m_{\Dt}|)\pi\m_{\Dt}, \vv}dt = -\lim_{\Dt\to 0}\int_0^T \intd{  \pi\nabla  u_{\Dt}, \vv }dt =- \int_0^T\intd{\nabla u, \vv}dt.
\eeqs
The proof is complete.
\end{proof}
\begin{lemma}\label{identeq1}
The following identity holds in $L^\infty(0,T; L^{ r}(\Omega))$ 
\beq\label{firsteq}
\phi \partial_t u^\lambda +\nabla\cdot \m = f. 
\eeq
Furthermore $u(x,0)=u_0(x)$, and $  u(x,T)=u_T$.
\end{lemma}
\begin{proof}
For $\varphi\in \mathcal D ([0,T])$ we defined the step function $\varphi_{\Dt}$ by 
\beqs
\varphi_{\Dt}(t) =\begin {cases}  
                        \varphi(t_{j-1}) &\text { if } t_{j-1}\le t< t_j, j=1,\ldots, J\\
                        \varphi(T)& \text { if }         t=t_J       
                      \end{cases}.
\eeqs
Using the test function $q=\psi\in\mathcal D(\Omega)$ in the second equation  of \eqref{semidiscrete-prob}, multiplying by $\Dt\varphi(t_{j-1})$ and summing up on $j=1,\ldots,J$, we obtain 
\beq\label{seo}
\sum_{j=1}^J \intd{ \phi\frac{ (u^j)^\lambda -  (u^{j-1})^\lambda}{h}, \psi}  \Dt\varphi(t_{j-1})  + (\nabla\cdot\m^j, \psi) \Dt\varphi(t_{j-1})
  =\sum_{j=1}^J (f^j, \psi) \Dt\varphi(t_{j-1}). 
\eeq 
Using the piecewise constant function $\pi$ this reads
\beqs
\int_0^T\intd{ \phi \partial_t \Pi u_{\Dt}^\lambda, \psi} \varphi_{\Dt} dt  + \int_0^T\intd{ \pi \nabla\cdot\m_{\Dt}, \psi} \varphi_{\Dt}  dt =\int_0^T\intd{ \pi f, \psi} \varphi_{\Dt}  dt.
\eeqs 
Since $\psi\varphi_{\Dt}$ converges strongly to $\psi\varphi$ in $L^1(0,T;R(\Omega))$. Hence a passage to the limit $\Dt\to 0$ implies that
\beq\label{seo2}
\int_0^T\intd{ \phi \partial_t u^\lambda, \psi} \varphi  dt  + \int_0^T\intd{ \nabla\cdot\m, \psi} \varphi dt =\int_0^T\intd{ f, \psi} \varphi dt.  
\eeq 
The set $\{  \psi\varphi, \psi\in \mathcal D(\Omega), \varphi\in \mathcal D ( \overline{(0,T)} ) \}$ is dense subset of $L^1(0,T; R(\Omega))$. Thus the identity \eqref{firsteq}  is established. 
 
 To prove the remaining two identities we rewrite \eqref{seo} as form
  \begin{multline*}
- \sum_{j=1}^J \Dt\intd{\phi (u^j)^\lambda, \psi}\frac{\varphi(t_j)- \varphi(t_{j-1})}{\Dt} + \sum_{j=1}^J \Dt\intd{ \nabla\cdot\m^j, \psi} \varphi(t_{j-1}) \\
=\sum_{j=1}^J \Dt\intd{ f^j, \psi }\varphi(t_{j-1})
-  \intd{ \phi (u^J)^\lambda, \psi}\varphi(T)+\intd{\phi (u^0)^\lambda, \psi}\varphi(0),
\end{multline*}
which is 
 \begin{multline*}
  -\int_0^T\intd {\phi \pi u_{\Dt}^\lambda, \psi} \partial_t \Pi\varphi dt  + \int_0^T\intd{ \pi \nabla\cdot\m_{\Dt}, \psi} \varphi_{\Dt} dt
   =\int_0^T\intd{\pi f, \psi} \varphi_{\Dt} dt-  \intd{\phi (u^J)^\lambda,\psi}\varphi(T)+\intd{\phi (u^0)^\lambda,\psi}\varphi(0).
 \end{multline*}
   Passing to the limit $\Dt\to 0$ we obtain
   \begin{multline}\label{ab0}
   -\int_0^T\intd{\phi u^\lambda, \psi} \partial_t \varphi dt  + \int_0^T\intd{ \nabla\cdot\m, \psi} \varphi dt=\int_0^T\intd{ f, \psi }\varphi dt-  \intd{ \phi (u_T)^\lambda, \psi}\varphi(T)+\intd{\phi(u^0)^\lambda, \psi}\varphi(0).
   \end{multline}
   On the other hand, partial integration of \eqref{seo2} yields
   \begin{multline}\label{ab1}
-\int_0^T\intd{\phi u^\lambda,\psi} \partial_t \varphi  dt  + \int_0^T\intd{\nabla\cdot\m, \psi} \varphi dt=\int_0^T\intd{ f, \psi} \varphi dt- \intd{\phi u^\lambda(T),\psi} \varphi(T)+ \intd{ \phi u^\lambda(0),\psi} \varphi(0).  
\end{multline} 
 We compare \eqref{ab0} and \eqref{ab1} to obtain 
 \beqs
 \intd{\phi (u^\lambda(0)-(u^0)^\lambda), \psi} \varphi(0)=\intd{\phi (u^\lambda(T)-u^\lambda_T), \psi} \varphi(T).
   \eeqs
   Since $\varphi(0)$ and $\varphi(T)$ are arbitrary, we have 
\beqs
\intd{\phi (u^\lambda(0)-(u^0)^\lambda ), \psi}=0=\intd{\phi (u^\lambda(T)-u^\lambda_T), \psi}.
   \eeqs
Thus $u(0)=u^0$ and $u(T)=u_T$.
  \end{proof}
  
\begin{lemma}\label{Fhat-Eq-F} The limit $\m$ of $\pi\m_{\Dt}$ and $\hat F$ of $ F(x,t, \pi\m_{\Dt})$ satisfy 
$\hat F = F(x,t,|\m|)\m$ in $L^\infty(0,T; (L^{s^*}(\Omega))^d )$. That means  
\beqs
\int_0^T ( \hat F,  \vv ) dt = \int_0^T  \intd{ F(x,t,|\m|)\m, \vv }dt \quad \text{ for all } \vv\in L^1(0,T; (L^s(\Omega))^d).
\eeqs
\end{lemma}  

To show this, we need an auxiliary result, a particular result of Lemma 1.2 in \cite{raviart69}.  
\begin{proposition}\label{aux1}
The limit $u$ of $\pi  u_{\Dt}$ satisfy 
\beq
\int_0^T  \left(\phi \partial_t u^\lambda(t) , u(t)\right) dt =\frac {1}{ r^*}   \int_\Omega\phi u^ r(T)   - \phi u^ r(0)  dx. 
\eeq
\end{proposition}
\begin{proof}
Equation \eqref{sum2eqs} rewrite as
\beqs
\intd{F(x,t,|\m^j|)\m^j , \m^j }+\intd{\phi\frac{ (u^j)^\lambda -  (u^{j-1})^\lambda}{h},  u^j} =\intd{f^j,  u^j}.
\eeqs

Since $ \intd{\phi\frac{ (u^j)^\lambda -  (u^{j-1})^\lambda}{h},  u^j} \ge \frac {1} {\Dt r^*} \intd{ \phi, (u^j)^{ r} -  (u^{j-1})^{ r} } $,
\beqs
\intd{ F(x,t,|\m^j|)\m^j, \m^j}+ \frac {1} {\Dt r^*} \intd{ \phi, (u^j)^{ r} -  (u^{j-1})^{ r} } \le \intd{f^j,  u^j}.
\eeqs
Multiplying $\Dt$ and summing up $j=1,\ldots, J$, we obtain  
\beqs
 \int_{0}^{T} \intd{ F(x,t,|\pi\m_{\Dt}|)\pi\m_{\Dt},\pi \m_{\Dt} }dt+\frac 1 { r^*} \intd{ \phi, (u^J)^{ r} -  (u^0)^{ r} } \le \int_{0}^{T} \intd{ \pi f_{\Dt}, \pi u_{\Dt} } dt.
\eeqs
We take the limit inferior to conclude that
\beqs
 \liminf_{\Dt\to 0}\int_{0}^{T} \intd{  F(x,t,|\pi\m_{\Dt}|)\pi\m_{\Dt},\pi \m_{\Dt} }dt+\frac {1} { r^*} \intd{ \phi, (u^J)^{ r} -  (u^0)^{ r} } \le \int_{0}^{T} \intd{ f, u } dt.
\eeqs
Using the result of Proposition~\ref{aux1}, we find that 
\beqs
 \liminf_{\Dt\to 0}\int_{0}^{T} \intd{ F(x,t,|\pi\m_{\Dt}|)\pi\m_{\Dt},\pi \m_{\Dt} } dt+  \int_0^T \intd{\phi \partial_t u^\lambda,u } dt   \le \int_{0}^{T} \intd{f, u} dt.
\eeqs
On the other hand, Lemma~\ref{identeq1} gives  
\beqs
\int_0^T\intd{\phi \partial_t u^\lambda, u} dt +\int_0^T \intd{\nabla\cdot \m, u} dt = \int_0^T \intd{f, u} dt. 
\eeqs
From \eqref{coincide2} and \eqref{coincide3} in Lemma~\ref{identeq0} we have 
\beqs
\int_0^T \intd{\nabla \cdot\m, u} dt = -\int_0^T \intd{ \m,\nabla u} dt =\int_0^T (\m,\hat F)dt.
\eeqs
 Therefore, 
 \beqs
  \liminf_{\Dt\to 0}\int_{0}^{T}  \intd{ F(x,t,|\pi\m_{\Dt}|)\pi\m_{\Dt},\pi \m_{\Dt} } dt\le \int_0^T (\m,\hat F)dt.
\eeqs
We have shown that for arbitrary $\vv\in L^\infty(0,T; (L^s(\Omega))^d )$
\beqs
\int_0^T(\hat F - F(x,t,|\vv|\vv), \m-\vv) dt  \ge \liminf_{\Dt\to 0} \int_0^T\intd{ F(x,t,|\pi\m_{\Dt}|)\pi\m_{\Dt} - F(x,t,|\vv|)\vv , \pi \m_{\Dt} -\vv } dt \ge 0. 
\eeqs
Choose $\vv= \m -\lambda  \varphi$, $\lambda>0$, $\varphi\in L^\infty(0,T; (L^s(\Omega))^d ) $ 
\beqs
\int_0^T(\hat F - F(x,t,|\m -\lambda  \varphi|)(\m -\lambda  \varphi), \lambda\varphi) dt \ge 0. 
\eeqs
Dividing $\lambda$ and letting $\lambda\to 0$, we obtain 
\beqs
\int_0^T(\hat F - F(x,t,|\m|)\m, \varphi dt \ge 0 \quad \text{ for all } \varphi\in L^\infty(0,T; (L^s(\Omega))^d ). 
\eeqs
This implies $\hat F = F(x,t,|\m|)\m$.  (see in the proof of Thm. 1.1 in \cite{raviart69} page 313).
\end{proof}

\begin{theorem}\label{ExistenceSol}
For all $f \in L^\infty(0, T; L^{ r}(\Omega)) $ that is Lipschitz continuous in time $t$. There exists  a  pair $(\m, u) \in L^\infty(0,T; W({\rm div}, \Omega))\times L^\infty(0,T; L^{ r}(\Omega)),$ such that
\begin{align*}
\int_0^T \intd{F(x, |\m|)\m, \vv} dt - \int_0^T \intd{\nabla \cdot \vv, u} =0\quad& \text{ for all } \vv \in L^1(0,T; L^{ r}(\Omega)),\\
\int_0^T \intd{\phi \partial_t u^\lambda, q}dt +\int_0^T \intd{\nabla\cdot \m, q} dt=\int_0^T\intd{f,q} dt \quad & \text{ for all } q \in L^1(0,T; L^{ r}(\Omega)).
\end{align*}
\end{theorem}
\begin{proof}
Let $\m$ be the limit of $\pi\m_{\Dt}$ and $u$ be the limit of $\pi u_{\Dt}$. Then Lemma~\ref{Fhat-Eq-F} and Lemma~\ref{identeq0} part (iii)  imply that
\beqs
\int_0^T \intd{F(x, t, |\m|)\m,\vv } dt = \int_0^T (\hat F,\vv ) dt =- \int_0^T (\nabla u,\vv ) dt= \int_0^T \intd{u, \nabla\cdot \vv }dt,
\eeqs
for all $ \vv \in L^1(0,T; L^{ r}(\Omega))$. In Lemma~\ref{identeq1} we have seen
that $(\m, u)$ fulfills the second equation.
\end{proof}



\myclearpage
\myclearpage
\appendix


\def\cprime{$'$} \def\cprime{$'$} \def\cprime{$'$}

\end{document}